\documentclass[12pt,twoside,reqno]{amsart}
\allowdisplaybreaks
\usepackage{bm}

\usepackage{amssymb,amsmath,mathrsfs,amstext,amsthm,amsfonts,amscd,xcolor}

\usepackage{bbm}
\usepackage{dsfont}

\usepackage{amsthm, enumerate}

\usepackage[ansinew]{inputenc}
\usepackage{graphicx}
\usepackage{hyperref}
\usepackage{tikz-cd}

\newcommand{\R}{\mathbb{R}}

\newcommand{\C}{\mathbb{C}}
\newcommand{\N}{\mathbb{N}}

\newcommand{\SL}{{\rm SL}}
\newcommand{\GL}{{\rm GL}}

\newcommand{\Pp}{\mathbb{P}}

\newcommand{\MM}{\mathcal{M}}

\theoremstyle{plain}
\newtheorem{theorem}{Theorem}[section]

\newtheorem{proposition}{Proposition}[section]
\newtheorem{corollary}[proposition]{Corollary}
\newtheorem{lemma}[proposition]{Lemma}
\theoremstyle{definition}

\newtheorem{definition}{Definition}[section]
\newtheorem*{theorem*}{Theorem}

\theoremstyle{definition}
\newtheorem{remark}{Remark}[section]

\numberwithin{equation}{section}

\newcommand{\abs}[1]{\left| #1 \right|} 
\newcommand{\norm}[1]{\left\|#1\right\|} 
\newcommand{\normtwo}[1]{
{\left\vert\kern-0.25ex\left\vert\kern-0.25ex\left\vert #1
    \right\vert\kern-0.25ex\right\vert\kern-0.25ex\right\vert} }

\newcommand{\ep}{\epsilon}

\newcommand{\om}{\omega}

\makeatletter
\newsavebox\myboxA
\newsavebox\myboxB
\newlength\mylenA

\newcommand*\xoverline[2][0.75]{%
    \sbox{\myboxA}{$\m@th#2$}%
    \setbox\myboxB\null
    \ht\myboxB=\ht\myboxA%
    \dp\myboxB=\dp\myboxA%
    \wd\myboxB=#1\wd\myboxA
    \sbox\myboxB{$\m@th\overline{\copy\myboxB}$}
    \setlength\mylenA{\the\wd\myboxA}
    \addtolength\mylenA{-\the\wd\myboxB}%
    \ifdim\wd\myboxB<\wd\myboxA%
       \rlap{\hskip 0.5\mylenA\usebox\myboxB}{\usebox\myboxA}%
    \else
        \hskip -0.5\mylenA\rlap{\usebox\myboxA}{\hskip 0.5\mylenA\usebox\myboxB}%
    \fi}
\makeatother

\newcommand{\Proj}{\mathbb{P}(\R^d)}

\newcommand{\Gr}{{\rm Gr}}

\newcommand\restr[2]{{
  \left.\kern-\nulldelimiterspace 
  #1 
  \vphantom{\big|} 
  \right|_{#2} 
  }}

\newcommand{\Prob}{\mathrm{Prob}}

\newcommand{\supp}{\mathrm{supp}}

\title[Analyticity of Lyapunov Exponents]{Analyticity of the Lyapunov exponents of random products of matrices}
\date{}

\begin{document}

\author[A. Amorim]{Artur Amorim}
\address{Instituto de Matem\'atica Pura e Aplicada  (IMPA), Brazil}
\email{artur.amorim@impa.br}

\author[M. Dur\~aes]{Marcelo Dur\~aes}
\address{Departamento de Matem\'atica, Pontif\'icia Universidade Cat\'olica do Rio de Janeiro (PUC-Rio), Brazil}
\email{accp95@gmail.com}

\author[A. Melo ]{Aline Melo}
\address{Instituto de Matem\'atica, Universidade Federal do Rio de Janeiro (UFRJ), Brazil}
\email{alinedemelo.m@gmail.com}

\begin{abstract}
This paper is concerned with the study of random (Bernoulli and Markovian) product of matrices 
on a compact space of symbols. We establish the analyticity of the maximal Lyapunov exponent as a function of the transition probabilities, thus extending the results and methods of Y. Peres from a finite to an infinite (but compact) space of symbols. Our approach combines the spectral properties of the associated Markov operator with the theory of holomorphic functions in Banach spaces.
\end{abstract}

\maketitle

\tableofcontents

\section{Introduction}\label{introv2}
Let $\Sigma \subset {\GL}_d(\R)$ be a compact subset and let $X := \Sigma^{\N}$ be the space of sequences $g = \{g_n\}_{n\in\N}$ on $\Sigma$. Denote by $\Prob (\Sigma)$ the set of probability measures on $\Sigma$.



Given $\mu \in \Prob (\Sigma)$, let $g = \{g_n\}_n$ be an independent and identically distributed (i.i.d) multiplicative process with law $\mu$. By Furstenberg-Kesten's theorem, there is a number $L_1 \in \R$, called the first (or maximal) \emph{Lyapunov exponent} of the process such that
$$
\frac{1}{n} \log \Vert g_{n-1}\cdots g_0 \Vert \to L_1 \quad \text{a.s.}
$$

Replacing the norm (or largest singular vale) of the product $g_{n-1}\cdots g_0$ by the other singular values, we obtain all the other Lyapunov exponents $L_2, \ldots, L_d$.



An important question in the study of Lyapunov exponents is concerned with its regularity. By this we mean the study of the continuity or the modulus of continuity of the top Lyapunov exponent with respect to the underlying data. This is a very subtle question that has motivated a lot of research over the last decades. It turns out that the choice of the topology will greatly influence the answer. 

We start by introducing Kifer's example, which shows that the Lyapunov exponent is, in general, not continuous. Consider the following matrices:
$$
A_1 = 
\begin{pmatrix}
2 & 0 \\
0 & \frac{1}{2}
\end{pmatrix} \quad \quad \quad A_2= 
\begin{pmatrix}
0 & -1 \\
1 & 0
\end{pmatrix}.
$$

We attribute to each of them the probabilities $p_1$ and $p_2$, respectively, such that $p_1+p_2=1$. Note that, if $p_1 =1$ and $p_2=0$, the Lyapunov exponent is equal to $\log 2$. However, when $p_2>0$, it is equal to $0$. 
Therefore, we may consider a sequence of cocycles where the matrices above are fixed and it is determined by the sequence of probabilities where $p_{2,n} = \frac{1}{2^n}$ and $p_{1,n} = 1-p_{2,n}$. Thus, we conclude that the Lyapunov exponent is not continuous. 
Moreover, one can consider the same example using the probability measure given by $\mu = p_1 \delta_{A_1} + p_2 \delta_{A_2}$. It is clear that continuity does not hold either relative to the weak* topology or on the total variation norm without further hypothesis.

Furstenberg, Kifer \cite{FK83} and Hennion \cite{Hen84} proved the continuity of the Lyapunov exponent with respect to the distribution $\mu$ in the weak* topology assuming an irreducibility condition, which is the absence of a subspace that is invariant for almost every matrix. Bocker and Viana \cite{BV} proved its continuity with respect to the topology induced by the weak* convergence plus the convergence relative to the Hausdorff distance of the supports without any irreducibility condition for random products of $\GL_2(\R)$ matrices. This result was generalized to the Markov case by Malheiro and Viana \cite{MaV} and then to fiber bunched cocycles by Backes, Brown and Butler \cite{BBB}. More recently, the result of Bocker and Viana was extended to $\GL_d(\R)$ matrices by Avila, Eskin, and Viana \cite{AvEV}.

Regarding quantitative results, Le Page \cite{LePage} proved H\"older continuous for a one-parameter family (under appropriate assumptions) which was generalized by Duarte and Klein \cite[Chapter 5]{DK-book}, where they proved that the Lyapunov exponent is a H\"older function of the cocycle. This result was then extended in different directions. Barrientos and Malicet \cite{BarMal25} managed to substitute the compactness by an exponential moment condition. Independently, Duarte and Graxinha \cite{DG25} also extended the result for the non compact but also non invertible case. Moreover, Duarte, Klein and Poletti \cite{DKP} extended it to the non locally constant case of H\"older linear cocycles with partially hyperbolic projective actions over hyperbolic dynamics.

The results presented in the previous paragraph assume some irreducibility hypothesis and are considered to be the generic case. The following results were proven in dimension $2$ and do not assume any kind of irreducibility.
Duarte and Klein \cite{DK-Holder} proved a weak H\"older modulus of continuity for the Lyapunov exponent
just assuming that $L_1$ is simple and that the underlying measure is finitely supported. Also in this setting,
Tall and Viana \cite{Tall-Viana} proved that, assuming just the simplicity hypothesis, the Lyapunov exponent is pointwise H\"older continuous and that even when the Lyapunov exponent is not simple, it is at least pointwise log-H\"older continuous.

On the other hand, the following example due to Halperin, which was then formalized by Simon and Taylor \cite[Appendix 3A]{ST85} provides an upper bound to the regularity of the Lyapunov exponent even in the generic case, and shows that the results of H\"older continuity are sharp. 
The example consists on a $1$-parameter family of measures on $\SL_2(\mathbb{R})$, given by 
$\mu_{a,b,E} := \frac{1}{2} \delta_{A_E} + \frac{1}{2} \delta_{B_E}$, where
\[
A_E = \begin{pmatrix}
a - E & -1 \\
1 & 0
\end{pmatrix}, \quad
B_E = \begin{pmatrix}
b - E & 0 \\
1 & 0
\end{pmatrix}.
\]

It follows from \cite[Theorem A.3.1]{ST85} that the map
$E \mapsto L(\mu_{a,b,E})$ is not $\alpha$-H\"older continuous for any 
\[
\alpha > \frac{2 \log 2}{\mathrm{arccosh}(1 + |a - b|/2)}.
\]

Note that the measures $\mu_{a,b,E}$ satisfy the assumptions of Le Page's theorem, which implies 
that the function $E \mapsto L(\mu_{a,b,E})$ is indeed H\"older continuous. However, the H\"older exponent $\alpha$ can be arbitrarily small as the difference $a - b$ is chosen to be very large. Moreover, one may consider the Wasserstein distance $W_1$, which metrizes the weak* topology, and compute the distance between two measures of this family. A straightforward computation shows that $W_1(\mu_{a,b,E_1}, \mu_{a,b,E_2}) \leq |E_1 - E_2|$. If we consider the same metric $\tau$ used in the work of Bocker and Viana, which consists of the Wasserstein metric added to the Hausdorff distance between the supports of the measures, the conclusion is the same, namely we have that $\tau(\mu_{a,b,E_1}, \mu_{a,b,E_2}) \leq 2|E_1 - E_2|$. This implies that the Lyapunov exponent cannot have a better regularity  than H\"older continuity neither for the weak* topology, nor for the weak* plus Hausdorff, even under the hypothesis of Le Page's theorem. Other works that also prove bounds for the regularity of the Lyapunov exponents include \cite{DKS19} by Duarte, Klein and Santos, \cite{BezDua23} by Bezerra and Duarte and \cite{BCDFK24} by Bezerra, Cai, Duarte, Freijo and Klein.

Moreover, a higher modulus of continuity was proved by Ruelle \cite{Rue79a}. He showed that the Lyapunov exponent restricted to the set of general uniformly hyperbolic cocycles is a real analytic function with respect to the fiber dynamics. On the context of random matrix products, Peres \cite{Pe91} showed that when $\Sigma$ is finite and $L_1$ is simple, the Lyapunov exponent is an analytic function of the probability weights. Bezerra, Sanchez, and Tall \cite{BST23} extended this result to random products of quasi-periodic cocycles. Recently, Baraviera and Duarte \cite{BD} obtained an extension of Peres in another direction. They proved that for a compact, but possibly infinite $\Sigma$, under quasi-irreducibility hypothesis, the Lyapunov exponent is Lipschitz with respect to the total variation norm.


The goal of this work is to extend the result of Peres to a broader scenario. By his work, it is natural to think that the dependence of the Lyapunov exponent on its underlying measure should be more regular as a function of its weight than as a function of its support. A natural setting to be considered is the one of measures that are absolutely continuous with respect to a reference measure. In this setting, we obtain in corollary \ref{analiticity Absolutely Continous} the analyticity of the Lyapunov exponent as a function of the densities.

By the example of Halperin, Simon and Taylor, it is not possible to obtain any modulus of continuity better than $\alpha$-H\"older for the Lyapunov exponents as a function of the underlying measure neither in the weak* topology nor on the weak* plus Hausdorff. Therefore, in order to obtain such a high regularity result, we must consider a stronger topology.

By putting together ideas of Peres, Baraviera and Duarte and tools of complex analysis in Banach spaces, we establish the analyticity of the top Lyapunov exponent with respect to the total variation norm. 
More precisely, we consider the space of complex measures endowed with the total variation norm, which is a Banach space. The subspace of complex measures with zero total variation is also Banach. Then, the set of complex measures with total variation equal to one is a translation of this Banach space which contains the set of probability measures. We prove that, in this affine Banach space, there is an extension of the Lyapunov exponent which is a holomorphic function. Therefore, we conclude that, when restricted to the probability measures, this extension is the Lyapunov exponent and it is a real analytic function. (Precise definitions of analyticity and total variation will be given in section \ref{preliminar}).

We state this result in two different settings.
In the first one, we assume a quasi-irreducibility hypothesis, which means the absence of proper invariant subspaces where the Lyapunov exponent restricted to it is not maximal. In other words, we assume that either there are no proper invariant subspaces or, if there exists such a subspace, then the Lyapunov exponent restricted to it is maximal.   
In the second one, instead of quasi-irreducibility, we assume that the probability measure has full support (which is an analogue of the assumption that each matrix has a positive probability in the finite support case of Peres \cite{Pe91}). 
More precisely, we establish the following.

\begin{theorem} \label{analiticity TV Bernoulli}
Let $\Sigma \subset {\GL}_d(\R)$ be a compact subset, $\mu_0 \in \Prob (\Sigma)$ and assume that $L_1(\mu_0) > L_2(\mu_0)$. 
\begin{enumerate}
\item[(1)] If $\mu_0$ is quasi-irreducible, then $\mu \mapsto L_1(\mu)$ is real analytic with respect to the total variation norm in a neighbourhood of $\mu_0$.
\item[(2)] If $\supp(\mu_0) = \Sigma$, then $\mu \mapsto L_1(\mu)$ is real analytic with respect to the total variation norm in a neighbourhood of $\mu_0$.
\end{enumerate}
\end{theorem}

\begin{remark}
Similar results hold for absolutely continuous measures and for random locally constant linear cocycles whose domain $\Sigma$ is an arbitrary compact set mapped to $\GL_d (\R)$ by a measurable and bounded function (see Section \ref{Corollaries}).
\end{remark}

\begin{remark}
In section \ref{Corollaries} we include an example where $\Sigma$ is not compact and the Lyapunov exponent is not even continuous. 
\end{remark}

\begin{remark}
	A natural question is if there exists any measure $\mu_0$, such that $L_1(\mu_0)> L_2(\mu_0)$, $\supp(\mu_0) \subsetneq \Sigma$ and $L_1$ is not analytic at $\mu_0$. The answer is Kifer's example with $\mu_0 = \delta_{A_1}$, that is, $p_1 =1$ and $p_2 = 0$, in which case the Lyapunov exponent is not even continuous. Therefore, we conclude that the simplicity of the Lyapunov exponent is not a sufficient condition for the result to hold in this generality.
	\end{remark} 
	

\begin{remark}

One may also ask how far is the total variation norm from being optimal. This is a very broad question. We analyse it this from the following point of view.
There is a natural way to define metrics on the space of measures, which is through the convergence of integrals in a given function space. In this context, the strength of a metric is related to the choice of the function space.

	Note that $\GL_d (\R)$ is a Polish space. A consequence of this fact, which follows from \cite[Theorem $6.9$]{Villani-OptimalTransport}, is that the following metric generates the weak* topology on the space of probability measures of bounded support:
	\begin{align*}
		&d_{\alpha} (\mu, \nu) := \sup_{v_\alpha(f)\le 1} \left \{\int f(x) d(\mu - \nu) (x) : f \in C^{\alpha} (\GL_d (\R), \R) \right \}
		\end{align*}
	where $0<\alpha\le 1$ and $v_\alpha(f):= \sup_{x \not = y} \frac{|f(x)-f(y)|}{|x-y|^{\alpha}}$. Another consequence is that the total variation norm can be characterized as:
	\begin{align*}
		&\norm{\mu -\nu}_{TV}:= \sup_{\norm{f}_{\infty} \le 1} \left \{\int f(x) d(\mu - \nu) (x) : f \in C^{0} (\GL_d (\R), \R) \right \}.
	\end{align*}

	We know that Theorem \ref{analiticity TV Bernoulli} holds under total variation norm and, from Halperin's example, that it does not hold for the weak* topology. We conclude that there is a family of metrics parametrized by $\alpha \ge 0$, whose strength decreases as $\alpha$ grows, such that the result holds for $d_0$ and fails to hold for $d_\alpha$ for every $\alpha>0$.
\end{remark}

It turns out that a similar statement holds for locally constant Markov linear cocycles, that is, linear cocycles over a Markov shift on a compact, possibly infinite space of symbols. Let us formally introduce the concept of Markov cocycles.

A Markov transition kernel $K \colon \Sigma \to \Prob (\Sigma)$ generalizes the concept of stochastic matrix that appears in the sub-shift of finite type. We will consider it to be continuous in the sense that given any Borel measurable set $E \subset \Sigma$, the map $x \mapsto K_x(E)$ is continuous.


A measure $\mu \in \Prob (\Sigma)$ is called $K$-stationary if for all Borel sets $E$, $\mu (E) = \int_\Sigma K_x (E) \, d \mu (x)$. 
The iterated Markov kernels $K^n$, $n\ge 1$ are defined inductively by $K^1 = K$ and $K^{n+1}_x (E) = \int_\Sigma K^n_y (E) \, d K_x (y)$, for all $x \in \Sigma$ and all Borel sets $E \subset \Sigma$. We assume that the kernel $K$ is {\em uniformly ergodic}, meaning that $K_x^n$ converges to $\mu$ uniformly (in $x \in \Sigma$) relative to the total variation distance. In this case the convergence is necessarily exponential (see \cite[Theorem 16.0.2]{meyn_tweedie_glynn_2009}) and the $K$-stationary measure $\mu = \mu_K$  is unique. This is the analogue of the stochastic matrix being primitive, in the finite setting. 
 
%






Let $K$ be a uniformly ergodic Markov kernel in $\Sigma$, $\mu$ be its unique stationary measure and let $\Pp_K = \Pp _{K, \mu_K}$ be the Markov measure on $X$ with initial distribution $\mu_K$ and transition kernel $K$. Given a continuous function $A \colon \Sigma \times \Sigma \to \GL_d (\R)$ and its extension to a locally constant fiber map on $X$, consider the corresponding skew product map $F_A \colon X \times \R^d \to X \times \R^d$, $F_A (\om, v) = (\sigma \om, A (\om_1, \om_0) v)$, where $\sigma$ denotes the forward shift.
 We regard $F_A$  as a linear cocycle over the Markov shift $(\sigma, \Pp_K)$ and refer to this dynamical system as a {\em Markov linear cocycle}. Its iterates are given by $F_A^n(\om,v) = (\sigma^n \om, A^n(\om)v)$, where
 $$
 A^n(\om)= A(\om_{n},\om_{n-1}) \dots A(\om_2,\om_1)A(\om_1,\om_0).
 $$ 
Similarly to the i.i.d product of matrices, there exists a number $L_1 \in \R$ which is called the top Lyapunov exponent of the process such that
$$
\frac{1}{n} \log \norm{A^n(\om)v} \to  L_1 \quad \text{a.s.}
$$


We fix $A$ and denote it by $L_1(K)$ to emphasize the dependence on the kernel $K$.

We call a transition kernel $K$ quasi-irreducible if the corresponding Markov cocycle $F_A$ is quasi-irreducible, meaning that there is no proper, invariant measurable section $\mathcal{V} \colon \Sigma \to \Gr(\R^d)$ such that the top Lyapunov exponent restricted to it is not maximal. $\Gr(\R^d)$ denotes the Grassmannian manifold of $\R^d$ and  an invariant section $\mathcal{V}$ is a map satisfying 
$A(\om_1, \om_0) \mathcal{V} (\om_0) = \mathcal{V} (\om_1)$. 



Given two kernels $K$ and $L$, we may consider a partial order relation between their supports. We say that $\supp ({L}) \subseteq \supp ({K})$ if $\supp ({L_\om}) \subseteq \supp ({K_\om})$ for every $\om \in \Sigma$.
For any fixed kernel $K$, we may define the set of continuous kernels whose supports are contained in the support of $K$ as the following: 
$$
\mathcal{S}(K) := \{ L \in \mathcal{K}(\Sigma) \colon \supp ({L}) \subseteq \supp ({K}) \}.
$$


Similarly to the i.i.d case, we consider the space of continuous complex kernels and endow it with the norm $\norm{K} := \sup_{\om \in \Sigma} \norm{K_\om}_{T.V.}$, such that it becomes a Banach space. We then consider an affine Banach subspace of it that contains the kernels such that $K_\om \in \Prob(\Sigma)$ for every $\om \in \Sigma$. We prove that there exists an extension of the Lyapunov exponent that is holomorphic, therefore the Lyapunov exponent is itself a real analytic function.

\begin{theorem} \label{analiticity TV Markov}
Let $K_0$ be a uniformly ergodic kernel on $\Sigma$ such that  $L_1(K_0) > L_2(K_0)$.
\begin{enumerate}
\item[(1)] If $K_0$ is quasi-irreducible, then $K \mapsto L_1(K)$ is real analytic  in a neighbourhood of $K_0$.
\item[(2)] The map $K \mapsto L_1(K)$ is real analytic  in a neighbourhood of $K_0$ in $\mathcal{S}(K_0)$.
\end{enumerate}
\end{theorem}

The main ideas in the proof of the Markov setting are analogous to the ones that appear in the i.i.d case, but we need to use tools from \cite{CDKM1} and \cite{DK-book} to deal with it. 

\begin{remark}
If we assume that the Lyapunov spectrum is simple, then item $2$ of both theorems holds for all the Lyapunov exponents. 

Additionally, if all exterior powers $\wedge_k g$, $1\le k < d$ of the sequence $g = \{g_n\}$ are quasi-irreducible, then item $1$ of both results also
holds for all the other Lyapunov exponents, as long as they are simple. 

%
\end{remark}

The rest of the paper is organized as follows. In Section \ref{preliminar} we introduce the concept of holomorphy in Banach spaces and a useful characterization of it, which 
will play a very important role in our proof. 
Moreover, we introduce the Markov operator  $Q_\mu$ and describe its general properties. 
In Section $3$ we study the convergence of the iterates of the Markov operator. When $\mu$ is a probability, there is a suitable observable, such that the iterates of the Markov operator applied to it converge to the top Lyapunov exponent $L_1$. More generally, when $\mu$ is any complex measure, we show that there is a domain where the limit function is holomorphic, which proves theorem \ref{analiticity TV Bernoulli}. 
In Section $4$ we extend the previous results to the Markov setting and prove theorem \ref{analiticity TV Markov}.


\section{Preliminaries}\label{preliminar}
This section is divided in two parts, both of which serve the purpose of constructing a setting that permits a generalization of Peres's arguments. 
The first part recalls the concept of analiticity in infinite dimensional Banach spaces as well as a useful criteria thereof.
The second part is devoted to the study of Markov operators defined by probabilities and the relationship with the convolution of measures. We define the average H\"older constant associated to a measure and derive a few estimates related to this quantity, which will be essential in chapter $3$. General references for these concepts are \cite{CB85} and \cite{Whitlessey}.

Throughout this chapter, $M$ and $N$ will denote Banach spaces over $\C$ and $U$ will denote an open subset of $M$. Moreover, whenever we write spaces such as $C^{\alpha} (\Proj), L^{1} (\mu)$ or $L^{\infty} (\mu)$, they should be understood as their complex version, unless stated otherwise. 

\subsection{Holomorphic functions in Banach spaces}


A function $f \colon U \to N$ is said to be \emph{holomorphic at a point} $a \in U$ 
if there exists $r>0$ and $n$-linear symmetric continuous map $T_n \colon M \times \dots \times M \to N$, for $n \in \N$ ($T_0$ is a constant) such that: 
$$
f(x) = \sum_{n=0}^\infty T_n(x-a)^n,
$$
for every $x \in B(a,r) \subset U$, where $T_n y^n$ denotes $T_n (y,y,\dots,y)$. If $f$ is holomorphic at every point of $U$, then $f$ is said to be \emph{holomorphic} on $U$.

If $M$ and $N$ are Banach spaces over $\R$ and $f:U \rightarrow N$ satisfy the definition above on the point $a \in U$, we will say that $f$ is \emph{real analytic at $a$}.

We introduce the following notation:
$$U(a,b) := \{ z \in \C \colon a+zb \in U  \} .$$

\begin{definition}
	A map $f \colon U \to N$ is said to be \emph{G\^ateux holomorphic} (or \emph{G-holomorphic}) if for every $a \in U$ and for every $b \in M$, the map
	$$
	z \mapsto f(a+zb)
	$$
	is holomorphic on $U(a,b) \subset \C$.
\end{definition}




It is clear that every holomorphic map is also $G$-holomorphic. However the converse in general is not true when $M$ is infinite dimensional. The following theorem provides a criteria for when the converse holds.

\begin{theorem} \cite[Chapter 14]{CB85} \label{Graves-Taylor-Hille-Zorn}
	Let $U$ be an open subset of a Banach space and let $f \colon U \to N$. The following are equivalent:
	\begin{enumerate}
		\item[(i)] $f$ is holomorphic on $U$.
		\item[(ii)] $f$ is G-holomorphic and continous on $U$.
	\end{enumerate}
\end{theorem} 

The \emph{variation of a complex measure} $\mu$ is defined as
$$ |\mu|(E) := \sup_{\pi} \sum_{A \in \pi} |\mu(A)|  $$
where the supremum is taken over all partitions $\pi$ of a measurable set $E$ into a countable number of disjoint measurable sets. An important property is that the variation of a complex measure is also a measure, whose values are defined on the set of nonnegative extended real numbers.

Another characterization of the variation of a complex measure is the following:

\begin{equation*}
	|\mu|(E) = \sup \left \{ \left |\int_{E} f(g) \; d\mu (g) \right | : f \in L^{\infty} (\mu) \text{ and } \norm{f}_{\infty} \le 1 \right \}.
\end{equation*}

Note that if $f \in L^1 (\mu)$ and $E$ is a measurable set, then:

\begin{equation*}
	\left | \int_E f(x) \; d\mu (x) \right | \le \int_E |f(x)| \; d|\mu| (x).
\end{equation*}


Let $\Sigma$ be a compact metric space. The total variation of a complex measure is defined as $\norm{\mu} \colon = |\mu|(\Sigma)$. If a measure satisfies $\norm{\mu} < \infty$, then we say that $\mu$ is finite or that is of bounded variation. In this case, $|\mu|$ is a finite nonnegative real valued measure.

We will consider $\Sigma$ to be a compact (but possibly infinite) space of symbols. We denote by $\MM (\Sigma)$ the set of complex valued measures over $\Sigma$ with bounded variation. The set $\MM(\Sigma)$, endowed with the total variation norm, will play the role of the Banach space $M$.

In this paper, we consider slightly more general definitions of holomorphy and G-holomorphy, in which the domain could also be a translation of a Banach subspace. The following construction shows how one can transfer the holomorphic structure from Banach spaces to affine subspaces via translation.

Let $V \subset M$ be a closed subspace, $v_0 \in M$ and consider a closed affine subspace $V_0 = V + v_0$ 
of $M$. Let $U_0 \subset V_0$ be an open set of $V_0$. We consider a function $f_0: U_0 \rightarrow N$ to be holomorphic (G-holomorphic) at $x_0 \in U_0$ if there exists a function $f\colon U=U_0 - v_0 \rightarrow N$ which is holomorphic (G-holomorphic) at $x_0 - v_0$, such that $f (x) = f_0 (x+v_0)$ for every $x\in U$. Moreover, if $f$ is holomorphic (G-holomorphic) in every point of its domain, then so is $f_0$. We will call $V_0$ an \textit{affine Banach space}.



It is then immediate that Theorem \ref{Graves-Taylor-Hille-Zorn} also holds in this context.

\subsection{Markov operators and convolution of measures}

%
%
%
%

Next, we present the definition of the convolution of measures in a more general setting for measures on groups.

\begin{definition}
	Let $G$ be a topological group that acts on a Polish space $S$. Let $\mu$ be a measure in $G$ and $\nu$ be a measure in $S$. Then we define the convolution of $\mu$ and $\nu$ as the measure $\mu*\nu$ on $S$ such that
	$$
	(\mu*\nu)(E) = \int_G \int_S \mathrm{1}_E(gx) \, d\nu(x)\, d\mu(g)
	$$for every measurable set $E \subset S$. By standard arguments of measure theory, in the same context,
	$$
	\int_S f(x) \, d(\mu*\nu)(x) = \int_G \int_S f(gx) \, d\nu(x)\, d\mu(g)
	$$
	for every measurable bounded function $f:S \rightarrow \C$.

\end{definition}

Given $k\geq 2$ and a measure $\mu \in \Prob(G)$, we define
$$
\mu^{*k} := \mu*\cdots *\mu \quad (k\, \text{times)}
$$
the $k$-th convolution of $\mu$ with itself.

Let us make three observations about the convolutions of measures. First, the convolution is associative, which makes $\mu^{*k}$ well defined. 
Second, it is a known result that the expression of the iterates of the cocycle $A^n (\omega)$ has distribution $\mu^{*n}$. Lastly, it should be noted that $\supp \; \mu^{*n} \not = \supp \; \mu^{*m}$ for $n \not= m$ in general. Indeed, even if $\mu = \delta_A$, for $A \in \GL_d (\R)$, we have that $\mu^{*k} = \delta_{A^k}$ for every $k \in \N$, which are generally different measures for different values of $k$. Instead, the convolution tends to ``spread out'' the support of the measure, making it contain matrices with arbitrarily big and small norms in most cases. In fact, it is clear from the definition of Lyapunov exponents in terms of convolutions that if the support of $\mu^{*n}$ is bounded from above (that is, it exists $K>0$ such that $\norm{A} \le K$ for every $A \in \supp \; \mu^{*n}$, for every $n \in \N$), then $L_1 (\mu) \le 0$. However, even though the support of a measure is not preserved through convolution, another important property holds: if the support of $\mu$ is a compact set in $\GL_d (\R)$, then $\supp \; \mu^{*n}$ is also a compact set in $\GL_d (\R)$ for every $n \in \N$ (moreover, the upper and lower bound of those sets are related). These observations will be used implicity throughout the text.

Let $\mu \in \MM(\Sigma)$. Consider the  operator $Q_\mu \colon L^\infty (\Proj) \to L^\infty (\Proj)$ associated to $\mu$ as follows:
$$
Q_\mu(\varphi)(\hat{v}) = \int_{\Sigma} \varphi(\hat{g}\hat{v}) \; d\mu(g),
$$
where $\hat{g}\colon \Proj \to \Proj$ is the projective action of $g$ defined by $\widehat{gv} =\hat{g}\hat{v}$ and $\Proj$ denote the projective space of dimension $d$ over $\R$. When $\mu \in \Prob(\Sigma)$, the operator $Q_\mu$ is called the Markov operator.
Note that the Markov operator $Q_\mu$ is a positive, bounded, linear operator that preserves constant functions. A general reference for this and related concepts in this section is \cite{DK-CBM}.

We may also consider small perturbations of $\mu \in \Prob(\Sigma)$ given by complex measures $\nu \in \MM(\Sigma)$ and their associated operators $Q_\nu$.
Although $Q_\mu$ is a Markov operator, $Q_{\nu}$ may not share the same properties. For instance, it does not fix constants when $\nu$ is not a probability measure. However, in this work we show that many results from the theory of Markov operators from \cite{BD} and \cite{DK-book} still hold for their complex analogues.

Moreover, we say that a measure $\eta \in \Prob(\Proj)$ is stationary with respect to $\mu$, or $\mu$-stationary, if it satisfies:
$$
\eta(B) = \int \eta(\hat{g}^{-1}(B))\, d\mu(g)
$$
for every measurable set $B \subset \Proj$. 
It turns out that a measure $\eta$  is $\mu$-stationary if, and only if, for every continuous function $\varphi \in C(\Proj)$,
$$
\int_{\Proj} Q_{\mu}(\varphi)\, d\eta = \int_{\Proj} \varphi\, d\eta.
$$

\begin{lemma} \label{Qop equality}
	Let $\mu \in \MM (\Sigma)$ be a complex measure with bounded variation. Then,
	$$
	Q_{\mu^{*n}} = Q^n_{\mu}.
	$$ 
\end{lemma}

\begin{proof}
	The proof proceeds by induction. 
\end{proof}

Given $\hat{p}, \hat{q} \in \Proj$, denote by $\delta\colon \Proj \times \Proj\to [0,\infty)$ the projective distance on $\Proj$:
\begin{equation} \label{projective distance}
\delta(\hat{p}, \hat{q}) := \frac{\|p \wedge q\|}{\|p\| \|q\|}
\end{equation}

\begin{definition} \label{def k alpha}
	Let $\mu \in \MM (\Sigma)$ be a complex measure of bounded variation and $0< \alpha< 1$. We define the average $\alpha$-H\"older constant of the projective action $\hat{g_0} \colon \Proj \to \Proj$ by
	$$
	k_\alpha(\mu) := \sup_{v_1 \neq v_2} \int_{\Sigma} \left ( \frac{\delta( \hat{g}_0\hat{v}_1  , \hat{g}_0\hat{v}_2  )}{\delta(\hat{v}_1, \hat{v}_2)} \right )^\alpha \; d| \mu |(g_0).
	$$
\end{definition}

Notice that, since $\supp \; \mu \subseteq \Sigma$, one can remove the domain of integration $\Sigma$, or choose any other compact set that contains the support of $\mu$ on the definition above, without affecting the value of $k_\alpha (\mu)$. This fact is important since we wish to understand the behaviour of $k_\alpha (\mu^{*n})$ as $n$ increases and, as previously stated, $\supp \; \mu^{*n} \not \subseteq \Sigma$ for large $n$.

Given $\mu, \nu \in \mathcal{M} (\Sigma)$, we denote $|\nu| \le |\mu|$ if $| \nu | (E) \le | \mu |(E)$ for every measurable set $E$.
%
%
%

\begin{lemma} For every two complex measures $\mu, \nu \in \MM (\Sigma)$, it holds that $| \mu*\nu | \leq | \mu |*|\nu|$. In particular, for every $\varphi \in L^\infty (\mu^{*n})$, we have that $\left | \int \varphi \; d\mu^{*n} \right | \leq \int |\varphi | \; d|\mu^{*n}| \leq \int |\varphi | \; d|\mu|^{*n}$.
\end{lemma}

\begin{proof}
	Let $\varphi \in L^{\infty} (\mu * \nu)$ and $E$ be a measurable set. Then:
	\begin{align*}
		\left |\int_{E} \varphi(x) \; d \mu*\nu (x) \right | &= \left | \int_{\Sigma} \int_{g^{-1}(E)} \varphi(gx) \; d \nu (x) d\mu (g) \right | \\
		&\le \int_{\Sigma} \int_{g^{-1}(E)} |\varphi(gx)| \; d |\nu| (x) d|\mu| (g) \\
		&= \int_{E} |\varphi(x)| \; d |\mu|*|\nu| (x).
	\end{align*}
	By restricting it to $\norm{\varphi}_{\infty} \le 1$ and taking the supremum on both sides, it follows that $|\mu * \nu|(E)\le |\mu |*|\nu| (E)$. Since $E$ is arbitrary, it follows that $|\mu*\nu| \le |\mu|*|\nu|$. Moreover, applying the inequality above multiple times with $\nu = \mu$ concludes the result.
\end{proof}

As a consequence of the previous lemma, if $\mu$ is a compactly supported measure of bounded variation, then $\mu^{*n}$ is also a compactly supported measure of bounded variation, for every $n \in \N$. In particular, the value $k_{\alpha} (\mu^{*n})$ is well defined for every $n \in \N$.


\begin{lemma} \label{sub multiplicative Kalpha}
	The sequence $k_\alpha(\mu^{*n})$ is sub-multiplicative:
	$$
	k_\alpha(\mu^{*(m+n)}) \leq k_\alpha(\mu^{*m})\, k_\alpha(\mu^{*n}).
	$$
\end{lemma}

\begin{proof}
	\begin{align*}
		&k_\alpha(\mu^{(m+n)}) = \sup_{v_1 \neq v_2} \int \left ( \frac{\delta( \hat{g}\hat{v}_1  , \hat{g}\hat{v}_2  )}{\delta(\hat{v}_1, \hat{v}_2)} \right )^\alpha \; d| \mu^{*(m+n)} |(g) \\
		&\le \sup_{v_1 \neq v_2} \int \left ( \frac{\delta( \hat{g}\hat{v}_1  , \hat{g}\hat{v}_2  )}{\delta(\hat{v}_1, \hat{v}_2)} \right )^\alpha \; d| \mu^{*n}|*|\mu^{*m} |(g) \\
		&= \sup_{v_1 \neq v_2} \int \int \left ( \frac{\delta( \hat{g_2} \hat{g_1}\hat{v}_1  , \hat{g_2} \hat{g_1}\hat{v}_2  )}{\delta(\hat{v}_1, \hat{v}_2)} \right )^\alpha \; d|\mu^{*m} |(g_1) d |\mu^{*n}| (g_2)\\
		&= \sup_{v_1 \neq v_2} \int \int \left ( \frac{\delta( \hat{g_2} \hat{g_1}\hat{v}_1  , \hat{g_2} \hat{g_1}\hat{v}_2  )}{\delta(\hat{g_1}\hat{v}_1, \hat{g_1}\hat{v}_2)} \cdot \frac{\delta( \hat{g_1}\hat{v}_1  , \hat{g_1}\hat{v}_2  )}{\delta(\hat{v}_1, \hat{v}_2)} \right )^\alpha d|\mu^{*m} |(g_1) d |\mu^{*n}| (g_2)\\
		&\le \sup_{v_1 \neq v_2} \int \left (\frac{\delta( \hat{g_1}\hat{v}_1  , \hat{g_1}\hat{v}_2  )}{\delta(\hat{v}_1, \hat{v}_2)} \right )^\alpha k_{\alpha} (\mu^{*n}) d |\mu^{*m}| (g_1)
		=k_{\alpha} (\mu^{*m}) k_{\alpha} (\mu^{*n}).
	\end{align*}

\end{proof}

Given $\varphi \in L^\infty(\Proj)$ and $0 \leq \alpha \leq 1$ we define
$$
v_\alpha(\varphi) := \sup_{\hat{v}_1 \neq \hat{v}_2}  \frac{| \varphi(\hat{v}_1) - \varphi(\hat{v}_2) |}{\delta( \hat{v}_1 , \hat{v}_2 )^\alpha}  .
$$

If $v_{\alpha}(\varphi) < \infty$ then $\varphi$ is $\alpha$-H\"older continuous. Let $C^\alpha(\Proj)$ be the space of all H\"older continuous functions, which we endowed with its natural norm $\norm{.}_\alpha = \norm{.}_\infty +v_\alpha(.)$. 

Note that $v_0 (\varphi)= \sup_{\hat{v}_1 \neq \hat{v}_2}  | \varphi(\hat{v}_1) - \varphi(\hat{v}_2) | \leq v_\alpha(\varphi)$ for every $\alpha >0$.

\begin{lemma} \label{Q contracts Valpha}
	For every $n \geq 1$, $\mu \in \MM (\Sigma)$ and $\varphi \in C^\alpha(\Proj)$, the following inequality holds:
	$$
	v_\alpha(Q_{\mu}^n(\varphi)) \leq k_\alpha(\mu^{*n})\, v_\alpha(\varphi).
	$$
\end{lemma}

\begin{proof}
	\begin{align*}
		\frac{| Q_{\mu}^n\varphi (\hat{v}_1) - Q_{\mu}^n\varphi (\hat{v}_2) |}{\delta(\hat{v}_1, \hat{v}_2)^\alpha}
		&= \frac{\left|  \int \varphi(\hat{g}\hat{v}_1)\; d\mu^{*n}(g) -  \int \varphi(\hat{g}\hat{v}_2)\; d\mu^{*n}(g) \right| }{\delta(\hat{v}_1, \hat{v}_2)^\alpha} \\
		&= \left| \int  \frac{  \varphi(\hat{g}\hat{v}_1) - \varphi(\hat{g}\hat{v}_2)  }{\delta(\hat{v}_1, \hat{v}_2)^\alpha} \; d\mu^{*n}(g) \right| \\
		& \leq \int \frac{   \left|  \varphi(\hat{g}\hat{v}_1) - \varphi(\hat{g}\hat{v}_2)  \right|    }{\delta(\hat{v}_1, \hat{v}_2)^\alpha} \; d | \mu^{*n} | (g)  \\
		& \leq  v_\alpha(\varphi)  \int \frac{    \delta( \hat{g}\hat{v}_1 , \hat{g}\hat{v}_2 )^\alpha     }{\delta(\hat{v}_1, \hat{v}_2)^\alpha} \; d | \mu^{*n} | (g).
	\end{align*}
	We conclude the lemma by applying the supremum in $\hat{v}_1 \neq \hat{v}_2$ to both sides.
\end{proof}

\section{Proof of Theorem \ref{analiticity TV Bernoulli}} \label{Bernoulli TV}
This section is divided into three parts. The first one is devoted to proving that the ideas in \cite{BD} still hold for complex measures. We show that the powers of the Markov operator $Q_\mu$ converges to a rank-one projection which, when $\mu$ is a probability measure, is the Lyapunov exponent. In the second part we use the concept of G\^ateux holomorphy to write the Markov operator as a polynomial. Therefore, we are able to use ideas of \cite{Pe91} to prove Theorem \ref{analiticity TV Bernoulli}. In the last part, we present some consequences of this result and we include one example that shows the importance of the compactness of the support of the measure.

\begin{subsection} {The convergence of the iterates of the Markov operator}

	
	
	A linear subspace $V \subset \R^d$ is called $\mu$-invariant if $g V = V$ for $\mu$-a.e. $g \in \GL_d(\R)$. One can restrict the cocycle to the subspace $V$ and consider the induced cocycle, with top Lyapunov exponent $L_1(\mu|_V)$. We call a probability measure $\mu$ {\em quasi-irreducible} if there is no proper $\mu$-invariant linear subspace $V \subset \R^d$ such that $L_1(\mu|_V) < L_1(\mu)$. 
	
	Note that the quasi-irreducibility of $\mu_0$ implies that Kifer's non random filtration (\cite[Chapter 3, Theorem 1.2]{Kifer-book}) is trivial, 
	that is, for every $v \in \R^d \backslash\{ 0 \}$ and $\mu_0^{\N}$-almost every $\{ g_n \}_n \in X$,
	$$
	\lim_{n\to \infty} \frac{1}{n}\log \norm{g_{n-1} \dots g_1g_0v} = L_1(\mu_0).
	$$
	
	Moreover, together with the hypothesis that $L_1(\mu_0) > L_2(\mu_0)$, a consequence of the previous fact is that 
	\begin{equation} \label{uniform convergence}
		\lim_{n\to \infty} \frac{1}{n} \int \log \norm{gv} d\mu_0^{*n}(g) = L_1(\mu_0),
	\end{equation}
	with uniform convergence in $\hat{v} \in \Proj$. For a proof of this consequence, see \cite[Proposition 5.2.2]{Duraes}.
	
	Note that, given $g \in \GL(d)$, $\norm{\wedge_2 g} = s_1 (g) s_2 (g)$, where $s_1(g)$ and $s_2(g)$ are the first and second singular values of a matrix.
	
	\begin{lemma} \label{desigualdade kalpha}
		For every $\mu \in \MM (\Sigma)$ and every $\alpha>0$,
		$$
		k_\alpha(\mu) \leq \sup_{ {\hat{v} \in \Proj}} \int_{\Sigma} \left({\frac{s_1( {g}) s_2( {g})}{\|{g} {v}\|^2}}\right)^\alpha \;d |\mu| (g).
		$$
	\end{lemma}
	
	\begin{proof} 
		
		

		
		Recall that
		$$
		\| gp \wedge gq \| \le s_1(g)s_2(g)\|p \wedge q\|.
		$$
		Hence, by (\ref{projective distance}), given $\alpha > 0$, two points $\hat{p}, \hat{q} \in \Proj$ and any $g \in \Sigma$, it holds that  
		\begin{align*}
			\left[ \frac{\delta(\hat{g} \hat{p}, \hat{g} \hat{q})}{\delta(\hat{p}, \hat{q})} \right] ^\alpha &\le \left[ (s_1(g)s_2(g))\frac{\|p\| \|q\|}{\|gp\| \|gq\|} \right]^\alpha \\
			&\leq \frac{\left[ s_1(g)s_2(g)\right]^\alpha}{2} \left[ \frac{1}{\|gp\|^{2\alpha}} +\frac{1}{\|gq\|^{2\alpha}}  \right]
		\end{align*}
where in the last inequality we use that the geometric mean is less or equal than the arithmetic mean.		
		
		Note that if we integrate with respect to the measure $|\mu|$ and take the supremum in $\hat{p} \neq \hat{q}$ on both sides of this inequality, we conclude the lemma.
	\end{proof}
	
	\begin{proposition} \label{contraction of Kalpha}
		Assume that $\mu_0 \in \Prob(\Sigma)$ is quasi-irreducible and $L_1(\mu_0)>L_2(\mu_0)$. Then, there exist $0< \alpha \leq 1$, $\theta >1$, $C>0$ and a neighbourhood $V \subset \MM (\Sigma)$ of $\mu_0$ with respect to the total variation distance, such that for every $n \in \N$ and for every $\mu \in V$,
		\begin{equation} \label{eq contraction of Kalpha Bernoulli}
			k_\alpha(\mu^{*n}) \leq C \theta^{-n}.
		\end{equation}
	\end{proposition} 
	
	\begin{proof}
		We follow the argument in \cite{DK-book}. By equation (\ref{uniform convergence})
		it holds that
		$$
		\lim_{n\to \infty} \frac{1}{n}\int \log \norm{g \, v}^{-2} d\mu_0^{*n}(g) = -2 L_1(\mu_0),
		$$
		with uniform convergence in $v\in \mathbb{S}^{d-1}$. 
		
		Hence choosing $\epsilon>0$ and choosing $n$ sufficiently large, we have that
		$$
		\int \log \Vert g \, v \Vert ^{-2} d\mu_0^{*n}(g) \leq n (-2L_1(\mu_0) +\epsilon) .
		$$
		
		Given a matrix $g\in\GL_d(\R)$, let $\wedge_2 g \in \GL_{d \choose 2} (\R)$ denote the second exterior power of $g$. Note that 
		$\norm{\wedge_2 g} = s_1 (g) \, s_2 (g)$.
		
		Then for $n$ large enough and for all $\om \in \Sigma$ we have
		\begin{align*}
			\frac{1}{n}\int \log (s_1(g) \,  s_2(g)) d\mu_0^{*n}(g) &= 
			\frac 1 n \, \int \log \norm{(\wedge_2 \, g)} d\mu_0^{*n}(g) \\
			& \le L_1 (\wedge_2 \, g, \mu_0) + \ep = 
			L_1(\mu_0) + L_2(\mu_0) + \ep \, .
		\end{align*}
		
		Choosing $\ep = \frac{1}{4}(L_1(\mu_0) - L_2(\mu_0))>0$ and combining the previous estimates, for all $v\in \mathbb{S}^{d-1}$, we get
		\begin{align*}
			\int \log \left( \frac{s_1( g) s_2( g)}{ \Vert gv \Vert ^{2}} \right) d\mu_0^{*n}(g) &\le n (L_1(\mu_0) + L_2(\mu_0) + \ep) + n \, (-2 L_1 (\mu_0)+\ep) \\
			& = - n \,  (L_1(\mu_0) - L_2(\mu_0)-2\ep) = \\
			&= - \frac{n}{2} \,  (L_1(\mu_0) - L_2(\mu_0)) < -1 ,
		\end{align*} 
		since $L_1(\mu_0)>L_2(\mu_0)$ and provided that $n$ is large enough.
		
		Using the inequality 
		$
		\exp x  \leq 1+x+\frac{x^2}{2} \exp \abs{x},
		$
		we conclude that for every $v \in \mathbb{S}^{d-1}$ and $n$ large enough,
		\begin{align*}
			\int \left({\frac{s_1( g) s_2( g)}{\|{g} {v}\|^2}}\right)^\alpha d\mu_0^{*n}(g)  &=  \int  \exp \left( \alpha \log  \frac{ s_1( g) s_2( g)}{ \Vert gv \Vert ^{2}} \right) d\mu_0^{*n}(g)   \\
			&\kern-5em \le   1+ \int  \left( \alpha \log \frac{s_1( g) s_2( g)}{ \Vert gv \Vert ^{2}} \right) d\mu_0^{*n}(g)  \\
			&\kern-5em + \int  \left( \frac{\alpha^2}{2}\log^2  \frac{s_1( g) s_2(g)}{ \Vert gv \Vert ^{2}} \exp \abs{\frac{ \alpha \log  s_1( g) s_2( g)}{ \Vert gv \Vert ^{2}}}  \right) d\mu_0^{*n}(g) \\
			&\kern-5em \le  1- \alpha +C \frac{\alpha^2}{2} 
		\end{align*}
		for some finite constant $C$ that depends only on $\mu_0$ and $n$.
		
		Thus, fixing $n_0$ sufficiently large and considering $\alpha$ small enough, we conclude by Lemma~\ref{desigualdade kalpha} that
		$$
		k_{\alpha}(\mu_0^{*n_0}) \leq 1 - \alpha +C \frac{\alpha^2}{2} < 1.
		$$
		
		Note that for a fixed $n$, the quantity  $\int \left({\frac{s_1( g) s_2( g)}{\|g {v}\|^2}}\right)^\alpha \, d|\mu_0|^{*n}(g)$ which bounds from above $k_\alpha (\mu_0^{*n})$, depends continuously on the measure.  
		Then the previous inequality extends to a neighbourhood of $\mu_0$. There exists $\kappa <1$ and a neighbourhood $V \subset \mathcal{M}(\Sigma)$ of $\mu_0$ with respect to the total variation distance, such that $k_\alpha(\mu^{*n_0}) \leq \kappa <1$ for every $\mu \in V$.
		
		Because of the sub-multiplicative property of $k_\alpha$, we conclude that there exists $C>0$ and $\theta >1$, such that inequality~\ref{eq contraction of Kalpha Bernoulli} holds for every $n \in \N$.
	\end{proof}

	\begin{corollary} \label{contraction of Valpha}
		Assume that $\mu_0 \in \Prob(\Sigma)$ is quasi-irreducible and $L_1(\mu_0)>L_2(\mu_0)$. Then there exists $0< \alpha \leq 1$, $\theta >1$, $C>0$ and a neighbourhood $V \subset \MM (\Sigma)$ of $\mu_0$ with respect to the total variation distance, such that for every $n \in \N$, for every $\mu \in V$ and every $\varphi \in C^\alpha(\Proj)$,
		$$
		v_\alpha(Q_{\mu}^n \varphi) \leq C v_\alpha(\varphi) \theta^{-n}.
		$$ 
	\end{corollary}
	
	\begin{proof}
		By proposition \ref{contraction of Kalpha}, there exist $0< \alpha \leq 1$, $\theta >1$, $C>0$ and a neighbourhood $V \subset \MM(\Sigma)$ of $\mu_0$ with respect to the total variation distance, such that for every $n \in \N$ and for every $\mu \in V$, we have that $k_\alpha(\mu^{*n}) \leq C \theta^{-n}$.
		Together with lemma \ref{Q contracts Valpha}, we conclude that
		$$
		v_\alpha(Q_{\mu}^n(\varphi)) \leq k_\alpha(\mu^{*n})\, v_\alpha(\varphi) \leq C v_\alpha(\varphi) \theta^{-n}.
		$$ 
	\end{proof}

		When $\mu$ is a probability measure on $\Sigma$, a consequence of corollary \ref{contraction of Valpha} is that there exist $\alpha \in (0,1]$, $\theta >1$ and $C<\infty$ such that for every $n \in \N$ and every $\varphi \in C^\alpha(\Proj)$,
		\begin{equation} \label{eq. strong mixing}
			\norm{Q_{\mu}^n \varphi - \int \varphi \; d\eta_{\mu}}_\alpha \leq C\theta^{-n} \norm{\varphi}_\alpha,
		\end{equation}
		where $\eta_\mu$ is the unique $\mu$-stationary measure on $\Proj$. That is because 
		\begin{align*}
			\norm{\varphi - \int \varphi \; d\eta}_{\infty} &= \left| \varphi(\hat{v}) - \int \varphi(\hat{p}) \; d\eta(\hat{p}) \right| \\
			&\leq \int| \varphi(\hat{v}) - \varphi(\hat{p})| \; d\eta(\hat{p}) \leq v_0(\varphi) \leq v_\alpha(\varphi).
		\end{align*}
		
		Therefore, since $\eta_\mu$ is $Q_\mu$ stationary,
		\begin{align*}
			\norm{Q^n_\mu \varphi - \int \varphi \;d\eta_\mu}_\infty &= \norm{Q^n_\mu \varphi - \int Q_\mu^n \varphi \;d\eta_\mu}_\infty 
			\leq \\
			&\leq v_\alpha(Q^n_\mu \varphi) \leq C v_\alpha (\varphi) \theta^{-n}.
		\end{align*}
		
		Moreover, 
		$$
		v_\alpha \left (Q^n_\mu \varphi - \int Q_\mu^n \varphi \;d\eta_\mu \right ) = v_\alpha(Q^n_\mu \varphi) \leq C\theta^{-n} \norm{\varphi}_\alpha,
		$$
		which concludes the argument.		
		
		It follows from the inequality (\ref{eq. strong mixing}) that $\eta_\mu$ is the unique $\mu$-stationary measure of this cocycle.
		
		Consider the observable $\varphi \colon  \Proj \to \R$ given by
		\begin{equation} \label{definition of varphi}
			\varphi(\hat{v}) = \int_{\Sigma} \log \frac{\norm{gv}}{\norm{v}} \; d\mu(g).
		\end{equation}

		If $\mu$ is a probability measure, then, by Furstenberg's formula,
		\begin{equation} \label{Furstenberg's formula}
			\int_{\Proj} \varphi(\hat{v}) \; d\eta_{\mu} = L_1(\mu).
		\end{equation}
		
		\begin{remark} \label{rmk lim markov op L1}
			When $\mu$ is a probability measure as above, for a fixed $\hat{v} \in \Proj$ the iterates $Q_\mu^n \varphi (\hat{v})$ converge uniformly to the top Lyapunov exponent $L_1(\mu)$.
		\end{remark}
		

	\end{subsection}

	\begin{subsection}{The domain of holomorphy} Now we establish a holomorphic extension of the Lyapunov exponent $L_1$. The way we approach this is the following: we will show that there exists an holomorphic function which depends on the measure $\mu$ and coincides with the Lyapunov exponent when $\mu$ is a probability. In particular, this fact guarantees that, under the hypothesis of the Theorem \ref{analiticity TV Bernoulli}, the usual Lyapunov exponent is real analytic, which is a property that does not depend on the extension chosen.
	

	We start by defining the domain where $L_1$ will be shown to be analytic.
		
			%
			%
			%
			%
			%
		
		Let $\MM_0(\Sigma)$ be the set of finite complex measures that give measure zero to $\Sigma$. Therefore, every $\mu \in \MM_0(\Sigma)$ satisfies $\mu(\Sigma)=0$.

		\begin{lemma}
			$\MM_0(\Sigma)$ is a Banach space.
		\end{lemma}
		
		\begin{proof}
			First note that $\MM_0(\Sigma)$ is a vector subspace of $\MM(\Sigma)$. Moreover, it is the kernel of the linear functional that assigns to each finite measure, its measure of the whole space: $\nu \mapsto \nu(\Sigma)$. Therefore, it is a closed subspace of $\MM(\Sigma)$, hence it is also a Banach space.
		\end{proof}
		
		Let $\MM_1(\Sigma)$ denote the set of finite complex measures that give measure one to $\Sigma$. Note that $\MM_1(\Sigma)$ is an affine Banach space of $\MM(\Sigma)$, namely $\MM_1(\Sigma)= \MM_0(\Sigma) + \mu_0$, for some $\mu_0 \in \Prob(\Sigma)$. Therefore $\MM_1(\Sigma)$ can be endowed with an analytic structure as seen in section $2.1$. 
		
		We are going to prove that $L_1$ admits a holomorphic extension to $V \cap \MM_1(\Sigma)$, where $V$ is the neighbourhood of $\mu_0$ from proposition \ref{contraction of Kalpha}. In fact, all the proofs from the previous sections were done in $\MM(\Sigma)$, but could have been done directly in $\MM_1(\Sigma)$. Thus, from now on, we are going to consider the neighbourhood $V$ to be in $\MM_1(\Sigma)$ and we will write just $V$ instead of $V \cap \MM_1(\Sigma)$.  
		
		\begin{lemma} \label{lemma 1 of uniform convergence}
			For every $\mu \in V \subset \MM_1(\Sigma)$, $\hat{v_1} \neq \hat{v_2} \in \Proj$ and $\varphi \in C^\alpha(\Proj)$ we have that  
			\begin{equation} \label{EqContracao1}
				\left |  Q^n_{\mu}\varphi(\hat{v}_1) - Q^n_{\mu}\varphi(\hat{v}_2)   \right| \leq C v_\alpha (\varphi) \theta^{-n}.
			\end{equation}

		\end{lemma}
		
		\begin{proof}
			By lemma \ref{Q contracts Valpha}, for every $\mu \in \MM(\Sigma)$ and every $\varphi \in C^\alpha(\Proj)$, we have $v_\alpha(Q^n_{\mu}\varphi) \leq v_\alpha(\varphi)k_\alpha(\mu^{*n})$.
			
			Therefore, it also holds that for every $\hat{v}_1 \neq \hat{v}_2 \in \Proj$,
			$$
			\left |  Q^n_{\mu}\varphi(\hat{v}_1) - Q^n_{\mu}\varphi(\hat{v}_2)   \right| \leq v_\alpha(\varphi)k_\alpha(\mu^{*n}).
			$$
			
			Then, by proposition \ref{contraction of Kalpha}, we conclude the proof.
		\end{proof}

		\begin{proposition} \label{proposition uniform convergence}
			For every $\mu \in V \subset \MM_1(\Sigma)$,$\hat{v_1} \neq \hat{v_2} \in \Proj$ and $\varphi \in C^{\alpha}(\Proj)$, 
			$$
			\left |  Q^{n+1}_{\mu}\varphi(\hat{v}_1) - Q^n_{\mu}\varphi(\hat{v}_2)   \right| \leq C v_\alpha (\varphi) \theta^{-n} \norm{\mu}.
			$$
		\end{proposition}
		
		\begin{proof}
			Note that for every $\hat{v}_1 \neq \hat{v}_2 \in \Proj$,
			$$
			\left| Q^{n+1}_{\mu} \varphi (\hat{v}_1) - Q^{n}_{\mu} \varphi (\hat{v}_2)    \right| = \left |  \int_\Sigma Q^n_{\mu}\varphi (\hat{g} \hat{v}_1) \; d\mu -  Q^{n}_{\mu} \varphi (\hat{v}_2)  \right |.
			$$
			
			Moreover, for every $\mu \in V \subset \MM_1(\Sigma)$ and $\hat{v_1} \neq \hat{v_2} \in \Proj$,  
			\begin{align*}
				&\left |  \int_\Sigma Q^n_{\mu}\varphi(\hat{g}\hat{v}_1) \; d\mu -  Q^{n}_{\mu} \varphi (\hat{v}_2)  \right | =  \left |  \int_\Sigma Q^n_{\mu}\varphi(\hat{g} \hat{v}_1)  -  Q^{n}_{\mu} \varphi (\hat{v}_2)  \; d\mu  \right | \leq \\
				&\leq  \int_\Sigma \left | Q^n_{\mu}\varphi(\hat{g} \hat{v}_1)  -  Q^{n}_{\mu} \varphi (\hat{v}_2) \right |  \; d\left |\mu \right | \leq  Cv_\alpha (\varphi)\theta^{-n} \norm{\mu}.
			\end{align*}
		\end{proof}
		
		Note that for a fixed $\hat{v}$ and $\varphi$ given in equation \ref{definition of varphi}, the sequence $\{ Q_\mu^n \varphi(\hat{v}) \}_n$ is Cauchy. Therefore its limit, denoted by $Q_\mu^\infty \varphi (\hat{v})$, exists. Also, by Lemma \ref{lemma 1 of uniform convergence}, this value does not depend of $\hat{v}$. As a consequence, after fixing $\varphi$, the limit $Q_{\mu}^{\infty} \varphi (\hat{v})$ can be seen as a complex valued function depending only on $\mu$. Moreover, note that $\mu \mapsto Q_\mu^n \varphi(\hat{v}) $ is continuous. Since $\mu \mapsto Q_\mu^\infty \varphi(\hat{v})$ is a uniform limit of continuous functions, it is also continuous. Furthermore, when $\mu$ is a probability measure, $Q_\mu^\infty \varphi(\hat{v}) = L_1(\mu)$, the top Lyapunov exponent (as shown in remark \ref{rmk lim markov op L1}).
		
		It is worth observing that, if we fix the measure $\mu$ and a direction $\hat{v}$, the expression $T_\mu \colon C^{\alpha} (\Proj) \rightarrow \C$ given by $T_\mu (\varphi) = Q_\mu^{\infty} \varphi (\hat{v})$ defines a continuous linear functional (which does not depend on $\hat{v}$). Therefore, by Riesz's theorem, there exists an unique complex measure $\eta_\mu$ in $C^{\alpha} (\Proj)$ such that
		\begin{equation*}
			Q_\mu^{\infty} \varphi = \int \varphi \; d\eta_{\mu}
		\end{equation*}
		for every $\varphi \in C^{\alpha} (\Proj)$. Moreover, if $\mu$ is a probability, then $\eta_{\mu}$ is a stationary measure for $\mu$.
		
		We want to prove that $\mu \mapsto Q_\mu^\infty \varphi(\hat{v})$ is holomorphic. For this, we are going to use theorem \ref{Graves-Taylor-Hille-Zorn}. Since we already know that the limit is continuous, it suffices to prove that it is also G-holomorphic. 
		
		As we stated, $\MM_1(\Sigma)$ is not a proper Banach space, therefore, we need to transfer the holormphic structure from $\MM_0(\Sigma)$ to it. Intuitively, G-holomorphy means to be holomorphic along complex lines, hence to say that the map $\mu \mapsto Q_\mu^\infty \varphi(\hat{v})$ from $ V \subset \MM_1(\Sigma)$ to $\C$ is G\^ateaux holomorphic means that $\forall \mu \in V$, $\forall \nu \in \MM_0(\Sigma)$, the map $z \mapsto Q_{\mu + z \nu}^\infty \varphi(\hat{v})$ is holomorphic on $V(\mu, \nu)=\{ z \in \C \colon \mu + z\nu \in V \}$. 
		
		Consider measures $\mu_z$ of the form $\mu_z= \mu + z\nu$, where $\mu \in  V$ and $\nu$ is any finite complex measure with $\nu(\Sigma)=0$. Note that, since $\mu \in \MM_1(\Sigma)$, we have that $\mu_z \in \MM_1(\Sigma)$ for every $z \in \C$. 
		Consider small perturbations of the Markov operator in the following sense: for each $z \in \C$, let the operator $Q_{\mu + z\nu} \colon L^\infty(\Proj, \C) \to L^\infty(\Proj, \C)$ be defined by
		\begin{align*}
			Q_{\mu + z\nu}(\varphi)(\hat{v}) &= \int_\Sigma \varphi(\hat{g}\hat{v})\; d(\mu + z \nu)(g) \\
			&= \int_\Sigma \varphi(\hat{g}\hat{v})\; d(\mu) + z\int_\Sigma \varphi(\hat{g}\hat{v})\; d(\nu) \\
			&=Q_{\mu}(\varphi)(\hat{v})+zQ_{\nu}(\varphi)(\hat{v}).
		\end{align*}
		
		
		Note that, for a fixed vector $\hat{v}$, each $Q_{\mu_z}^n(\varphi)(\hat{v})$ is a polynomial of degree smaller or equal to $n$, in particular, the map $z \mapsto Q_{\mu_z}^n(\varphi)(\hat{v})$ is holomorphic for $\mu_z \in V$.
		Therefore, for every $z \in V(\mu, \nu)$, the limit function is a uniform limit of holomorphic functions, hence $z \mapsto Q_{\mu_z}^\infty \varphi(\hat{v})$ is holomorphic. In other words, the Lyapunov exponent is G-holomorphic in the neighbourhood $V \subset \MM_1(\Sigma)$ of $\mu_0$. Together with the continuity, we conclude that it is indeed holomorphic. This concludes the proof of Theorem \ref{analiticity TV Bernoulli} item $(1)$.


Now we prove item $(2)$ of theorem \ref{analiticity TV Bernoulli}. We drop the assumption of irreducibility of $\mu_0$ and instead we assume that $\supp \, \mu_0 = \Sigma$.
Let $W$ be a non trivial vector subspace of $\R^d$ that is invariant for $\mu_0$ almost every matrix $g$. The measure $\mu_0$ defines the measures $ \mu_{0,W}$ and $\mu_{0,\R^d / W}$ in $\GL_d(W)$ and $\GL_d(\R^d / W)$. Moreover, they induce the linear cocycles restricted to $W$ and to $\R^d /W$ , with Lyapunov exponents $L_1(\mu_{0,W})$ and $L_1(\mu_{0,\R^d /W})$.
		
By lemma $3.6$ of \cite{FK83}, $L_1(\mu_0) = \max \{ L_1(\mu_{0,W}), L_1(\mu_{0, \R^d / W}) \}$.

		Without loss of generality we may suppose that $L_1(\mu_0) = L_1(\mu_{0,W})$. The other case is similar. We claim that, in fact, $L_1(\mu) = L_1(\mu_W)$ for every $\mu$ in a neighbourhood of $\mu_0$.
		
		First note that $\supp \, \mu_0 = \Sigma$ implies that $gW =W \; \forall g \in \Sigma$. Therefore, every $\mu \in \Prob(\Sigma)$ also satisfies $gW=W$ for $\mu$-a.e $g$. Hence, it holds that $L_1(\mu) = \max\{  L_1(\mu_W), L_1(\mu_{\R^d / W}) \}$. Moreover, by corollary $B$ of \cite{Pe91}, if $\mu_n \to \mu_0$ in the weak star topology and $\supp \, \mu_n \subset \supp \, \mu_0$ for every $n$, then $L_1(\mu_n) \to L_1(\mu_0)$. 

		
		Since $\supp \, \mu_W \subseteq \supp \, \mu_{0,W}$ and $\supp \, \mu_{\R^d / W} \subseteq \supp \, \mu_{0,\R^d /W}$, we may use the same corollary to each induced cocycle to conclude that both $\mu \mapsto L_1(\mu_{W})$ and $\mu \mapsto L_1(\mu_{\R^d / W})$ are continuous maps. By assumption, $L_1(\mu_0) = L_1(\mu_{0,W}) > L_1(\mu_{0, \R^d /W})$. Thus, $L_1(\mu) = L_1(\mu_{W}) > L_1(\mu_{\R^d /W})$ for every $\mu$ sufficiently close to $\mu_0.$ This concludes the proof of the claim. 
		
		Note that the Lyapunov spectrum of both $\mu_{0,W}$ and $\mu_{0,\R^d / W}$ are contained in the Lyapunov spectrum of $\mu_0$. Since $L_1(\mu_0) $ is simple, the same holds for $L_1(\mu_{0,W})$. Hence, if $\mu_{0,W}$ is irreducible, then the map $\mu \mapsto L_1(\mu_W)$ is holomorphic. Since $L_1(\mu) = L_1(\mu_W)$ in a neighbourhood of $\mu_0$, we conclude that $\mu \mapsto L_1(\mu)$ is also holomorphic in a neighbourhood of $\mu_0$.
		
		If $ \mu_{0,W}$ is not irreducible, there exists another non trivial invariant subspace $W' \subset W$. The measure $\mu_{0,V}$ defines measures $\mu_{0,W'}$ and $\mu_{0,W / W'}$. Then we do the same procedure. Since the invariant subspaces are of decreasing dimension, this process must stop after a finite number of steps. Therefore, we conclude the proof of Theorem \ref{analiticity TV Bernoulli}.
		
	\end{subsection}
	
	\begin{subsection}{Corollaries and remarks} \label{Corollaries}
		
		
		Let $\Sigma$ be an abstract compact space, $X = \Sigma^{\N}$ and $\sigma \colon X \to X$ be the forward shift on $X$. We fix a measurable and bounded function $A\colon \Sigma \to \GL_d(\R)$ and denote also by $A$ the locally constant (fiber) map $A \colon X \to \GL_d(\R)$ given by $A((x_n)_{n \in \N})=A(x_0)$.
		
		Given $\mu \in \Prob (\Sigma)$, let $\mu^\N$ be the product (Bernoulli) measure on $X$. A random (Bernoulli) locally constant linear cocycle $F_A \colon X \times \R^d \to X \times \R^d$ relative to the product measure $\mu^\N$ is a skew product transformation such that
		$$
		F_A(\om,v) = (\sigma(\omega),A(x)v).
		$$
		Its iterates are given by 
		$$
		F_A^n(\om,v)=(\sigma^n(\om),A^n(\om)v),
		$$
		where $A^n(\om):=A(\om_{n-1}) \dots A(\om_1)A(\om_0)$.
		
		A seminal result from Furstenberg and Kesten states that under the integrability condition $\log^+\norm{A^{\pm}} \in L^1(\mu)$, the limit
		$$
		L_1(A, \mu) = \lim_n \frac{1}{n}\log \norm{A^n(\om)}
		$$
		exists $\mu$ a.e. and it is called the top Lyapunov exponent of this cocycle. Consider the push forward measure on $\Prob(\GL_d(\R))$ given by $A_*\mu$. By the boundedness of $A$, the support of $A_*\mu$ remains compact, and therefore its Lyapunov exponent is well defined. A straightforward computation shows that the Lyapunov exponent $L_1 (A_*\mu)$ associated to the measure $A_*\mu$ is equal to $L_1 (A, \mu)$. Moreover, the application $A_* : \mathcal{M} (\Sigma) \rightarrow \mathcal{M} (\GL_d (\R))$ is a linear continuous (and, therefore, analytic) mapping that preserves probabilities. Since the composition on analytic maps is analytic, it follows by Theorem \ref{analiticity TV Bernoulli} that the map $L_1 (A, \mu)= L_1 (A_* \mu)$ is analytic with respect to $\mu$, which guarantees that the result holds for arbitrary locally constant linear cocyles. 
		
		
		
		A second corollary is an analogue of the finite case, in which the support is of the measure is a fixed compact set on $\GL_d (\R)$ and we look at the dependence on the probability weights. 
		Let $\mu_0 \in \Prob(\Sigma)$ be a reference measure of full support. We restrict to the measures in $\Sigma$ which are absolutely continuous with respect to $\mu_0$.  
		
		By the Radon-Nikodym Theorem, this space is identified with the space $L^1 (\mu_0)$ of integrable complex functions with respect to $\mu_0$ through the map $I: L^1 (\mu_0) \rightarrow \mathcal{M} (\Sigma)$ given by
		$$
		I(f)(E)= \int_E f(x) d \mu_0 (x)
		$$
		for every measurable set $E$. The map $I$ is an isomorphism between $L^1 (\mu_0)$ and the measures on $\mathcal{M} (\Sigma)$ which are absolutely continuous with respect to $\mu_0$. Observe that, given $f,g \in L^1 (\mu_0)$, it follows that
		$$
		\norm{I(f)-I(g)}_{TV} \le \norm{f-g}_{1} \le \norm{f-g}_p,
		$$
		where $\norm{.}_{TV}$ denotes the total variation norm, $\norm{.}_1$ denotes the $L^1$ norm and $\norm{.}_p$ denotes the $L^p$ norm, with $p \in [1,+\infty]$. This fact guarantees that, given $r>0$, it follows that $B_p(f, r) \subset B_1 (f,r) \subset B_{TV}(I(f), r)$, where each of the previous sets denotes an open ball on its respective norm.
		
		This observation, aligned with the Theorem \ref{analiticity TV Bernoulli}, proves the following.
		
		\begin{corollary} \label{analiticity Absolutely Continous}
			Let $\mu_0 \in \Prob (\Sigma)$ have full support and assume that $L_1(\mu_0) > L_2(\mu_0)$. For $p \in [1,+\infty]$, define 
			$$
			L^p_1 (\mu_0):= \left \{ f:\Omega \rightarrow \R : f \in L^p (\mu_0) \text{ and } \int f (x) d\mu_0 (x)=1 \right \}.
			$$	
			Then, the Lyapunov exponent $L_1: L^p_1 (\mu_0) \rightarrow \R$ is a real analytic function of $f$ with respect to the $L^p$ norm in a neighbourhood of the constant function $1$.
		\end{corollary}
		
		We now regard the set in which $L_1$ is analytical. We say that a measure $\mu \in \mathcal{M} (\GL_d (\R))$ is \emph{irreducible} if there is no proper subspace $V \subset \GL_d (\R)$ such that $gV = V$ for $\mu$-a.e.$g$. Notice that every irreducible measure is quasi-irreducible.
		
		Observe that irreducibility is a dense property with respect to the total variation norm. Indeed, let $\mu_0$ be a probability in $\GL_d(\R)$, and let $\nu$ be another probability in $\GL_d(\R)$, with compact support. If $\nu$ is irreducible and given $\varepsilon > 0$, then $\mu_{\varepsilon} = (1-\varepsilon)\mu + \varepsilon \nu$ is an irreducible probability in $\GL_d(\R)$. To see this, let $V$ be a proper subspace of $\GL_d(\R)$. Since $\nu$ is irreducible, there exists a borelian set $B \subset \GL_d(\R)$ such that $\nu (B)>0$ and $gV \not = V$ for every $g \in B$. Notice that $\mu_{\varepsilon} (B) = (1-\varepsilon) \mu (B) + \varepsilon \nu (B) \ge \varepsilon \nu (B) > 0$, so it follows that $\mu_{\varepsilon}$ is irreducible.
		
		Notice also that $\norm{\mu_{\varepsilon} - \mu} = \norm{\varepsilon \nu - \varepsilon \mu} \le 2 \varepsilon$, so we can choose $\varepsilon$ sufficiently small such that $\mu_{\varepsilon}$ is arbitrarely close to $\mu$. Moreover, $\supp \; \mu_{\epsilon} = \supp \; \mu \cup \supp \; \nu$, so if $\mu$ has compact support, $\mu_{\epsilon}$ also has compact support, and if $\supp \; \mu, \supp \; \nu \subset \Sigma$, then $\supp \; \mu_{\epsilon} \subset \Sigma$. 
		
		We also observe that, on \cite{Kifer1982-ce}, Kifer proved that being irreducible was an open property on $\Prob(\GL_d(\R))$ with respect to the weak* topology. Since the total variation norm generates a finer topology than the weak* topology, it follows that being irreducible is also an open property with respect to the total variation norm. Therefore, by Theorem \ref{analiticity TV Bernoulli}, we can conclude that $L_1$ is analytic on the set of compactly supported irreducible measures on $\GL_d (\R)$, which is a dense open set on the space $\Prob (\GL_d(\R))$ with respect to the total variation norm.
		
		To conclude this section, we make the observation that the restriction of the probabilities to a compact set $\Sigma$ in Theorem \ref{analiticity TV Bernoulli} cannot be removed. Indeed, let $\mu \in \Prob(\GL(d))$ be a compactly supported measure with $L_1 (\mu) > L_2 (\mu)$. Let $a, \varepsilon>0$ and consider the measure $\mu_{a,\varepsilon}:=(1-\varepsilon) \mu + \varepsilon \delta_{aI}$, where $I$ is the identity matrix. By a previous comment, the measure $\mu_{a, \varepsilon}$ is a compactly supported measure such that $\norm{\mu_{a,\varepsilon} - \mu} < 2\varepsilon$. The identity:
		\begin{equation*}
			L_1 (\nu) + ... + L_d (\nu) = \int \log |\det g| \; d \nu (g),
		\end{equation*}
		which is true for any compactly supported measure $\nu$ on $\GL(d)$, gives us:
		\begin{equation*}
			L_1 (\mu_{a, \varepsilon}) \ge \dfrac{L_1 (\mu_{a, \varepsilon}) + ... + L_d (\mu_{a, \varepsilon})}{d} = \dfrac{1}{d}\int \log |\det g| \; d \mu_{a, \varepsilon} (g).
		\end{equation*}
		
		Now, given $\delta > 0$, let $0<\varepsilon< \delta/2$, so that $\mu_{a, \varepsilon} \in B(\mu, \delta)$ for every $a>0$. Observe that:
		$$
		\dfrac{1}{d}\int \log |\det g| \; d \mu_{a, \varepsilon} (g) = \varepsilon \log a+\dfrac{(1-\varepsilon)\int \log |\det g| \; d \mu (g)}{d}.
		$$
		Then, choose $a$ sufficiently big, such that:
		\begin{equation*}
			\log a > \dfrac{L_1 (\mu) +1 + \dfrac{\varepsilon-1}{d}\int \log |\det g| \; d \mu (g)}{\varepsilon}.
		\end{equation*}
		The previous inequality then guarantees that $L_1 (\mu_{a, \varepsilon}) > L_1 (\mu) + 1$. In particular, $L_1$ cannot be continuous in $\mu$, much less analytic. The problem relies on the fact that, as we shrink the neighborhood of $\mu$, the value of $\varepsilon$ decreases, which causes the choice of $a$ above to become increasingly larger. Restricting ourselves to a compact set limits the size of $a$, which make this construction fail for small enough $\delta$.
		%
		%
		
	\end{subsection}
	

\section{Analytic dependence on Markovian transitions}
\label{Markov TV}
The goal of this section is to prove Theorem \ref{analiticity TV Markov}, which is the Markovian analogue of Theorem \ref{analiticity TV Bernoulli}. A Markov kernel $K \colon \Sigma \to \Prob(\Sigma)$ gives the transition probabilities of the dynamics and is the natural generalization of the concept of stochastic matrix for sub-shifts of finite type. 
Similarly to the previous section, where we considered complex valued measures, we are going to consider complex Markov kernels $K \colon \Sigma \to \MM(\Sigma)$.

We denote by $\mathcal{K}(\Sigma)$ the set of continuous and complex Markov kernels over $\Sigma$ such that for every $\om \in \Sigma$, $K_\om(\Sigma)$ has bounded variation. We denote by $\mathcal{K}_{\Prob}(\Sigma)$ the set of (continuous) Markov kernels $K$ over $\Sigma$, such that for every $\om \in \Sigma$, $K_{\om} \in \Prob(\Sigma)$ and $K_\om$ has bounded variation.

Consider the following norm in $\mathcal{K}(\Sigma)$:
$$
\norm{K} := \sup_{\om \in \Sigma} \norm{K_\om}_{T.V.}.
$$

\begin{proposition}
The set $\mathcal{K}(\Sigma)$ endowed with the norm $\norm{.}$ is a Banach space.
\end{proposition}

\begin{proof}
It is clear that $\mathcal{K}(\Sigma)$ is a normed vector space.  Note that if $\{ {K_n} \}_n$ is a Cauchy sequence in $(\mathcal{K}(\Sigma), \norm{.})$, then for every $\om \in \Sigma$, $K_{\om,n}$ is also Cauchy. Since $(\MM (\Sigma), \norm{.}_{TV})$ is complete, for each $\om$, $(K_{\om, n})_n$ converges to some complex measure in $\MM (\Sigma)$. Hence $(K_n)_n \to K^*$ which is defined as $K^*_\om := \lim_n K_{\om , n}$.
\end{proof}

%

Given a continuous Markov Kernel 
$K \in \mathcal{K}_{\Prob}(\Sigma)$, we can define the Markov operator $Q_{K} \colon L^\infty (\Sigma \times \Proj) \to L^\infty (\Sigma \times \Proj)$ as follows:

$$
Q_{K}(\varphi)(\om_0, \hat{v}) = \int_{\Sigma} \varphi(\om_1, A(\om_1, \om_0)v) \; dK_{\om_0}(\om_1).
$$

The Markov operator $Q_{K}$ is a positive, bounded, linear operator that preserves constants.

\begin{definition}
We say that a Markov kernel $K \in \mathcal{K}_{\Prob}(\Sigma)$ is \emph{uniformly ergodic} if $K_\om^n$ converges uniformly in $\om$ to its stationary measure $\mu$ with respect to the total variation norm.
\end{definition}

Moreover, $K$ being uniormly ergodic is equivalent to the existence of constants $C<\infty$ and $\theta>1$ such that 
$$
\norm{Q_K^n \varphi - \int \varphi \;d\mu}_\infty \leq C\theta^{-n} \norm{\varphi}_\infty \quad \forall \varphi \in L^\infty(\Sigma \times \Proj), \; \forall n \in \N.
$$

Similar to i.i.d. case, we may also consider small perturbations of $K \in \mathcal{K}_{\Prob}(\Sigma)$ given by complex valued Markov kernels $L \in \mathcal{K}(\Sigma)$ and their associated operators $Q_{L}$.
Although $Q_{K}$ is a Markov operator, $Q_{L}$ may not be a Markov operator. This can happen because when $L_{\om}$ is not a probability measure, for some $\om \in \Sigma$, it does not fix the constant functions.

Given $0 \le \alpha \leq 1 $ and $\varphi \in L^\infty(\Sigma \times \Proj)$, we define the H\"older seminorm $v_\alpha$, the H\"older norm $\| \cdot \|_\alpha$ and the space $C^\alpha(\Sigma \times \Proj)$ of H\"older continuous observables (in the projective variable)  by:
\begin{align*}
&v_\alpha(\varphi) = \sup_{ \substack  {{\omega_0 \in \Sigma} \\    \hat{v}_1 \neq \hat{v}_2}  } \frac{| \varphi(\omega_0,\hat{v}_1) - \varphi(\omega_0,\hat{v}_2) |}{\delta(\hat{v}_1,\hat{v}_2)^\alpha}, \\
&\| \varphi \|_\alpha = v_\alpha(\varphi) + \| \varphi \|_\infty, \\
&C^\alpha(\Sigma \times \Proj) = \{ \varphi \in L^\infty(\Sigma \times \Proj) \colon \|\varphi\|_\alpha < \infty \}.
\end{align*}

Moreover, given a Markov kernel $K$, we consider $\Pp_{\om_0}$ the complex analogue of a Markov measure with initial distribution $\delta_{\om_0}$ and transition given by the Markov kernel $K$. A construction for this, in the real case, can be found in \cite[Chapter 5]{DK-book} (the complex case is analogous). Note that, if $K \in \mathcal{K}_{\Prob} (\Sigma)$, then $\Pp_{\om_0}$ is a proper Markov measure.
We consider the Markovian analogue of definition $\ref{def k alpha}$, the average $\alpha$-H\"older constant:   

\begin{align*}
k_\alpha(F_{A,K}) &= \sup_{\substack  {{\omega_0 \in \Sigma} \\    \hat{v}_1 \neq \hat{v}_2}} \int \frac{\delta(\hat{A}(\omega)\hat{v}_1, \hat{A}(\omega)\hat{v}_2)^\alpha}{\delta(\hat{v}_1,\hat{v}_2)^\alpha} \;d|\Pp_{\omega_0}|(\omega) \\
&= \sup_{\substack  {{\omega_0 \in \Sigma} \\    \hat{v}_1 \neq \hat{v}_2}} \int \frac{\delta(\hat{A}(\omega_1,\omega_0)\hat{v}_1, \hat{A}(\omega_1,\omega_0)\hat{v}_2)^\alpha}{\delta(\hat{v}_1,\hat{v}_2)^\alpha} \;d|K_{\omega_0}|(\omega_1).
\end{align*}

We claim that the same properties that we proved in the Bernoulli case, also hold for Markov cocycles. We prove the analogous of lemma \ref{Q contracts Valpha}.


\begin{lemma} \label{Q contracts Valpha Markov}
For every $n \geq 1$, $K \in \mathcal{K} (\Sigma)$ and $\varphi \in C^\alpha(\Sigma \times \Proj)$, the following inequality holds:
$$
v_\alpha(Q_{K}^n(\varphi)) \leq k_\alpha(F_{A^n,K^{n}})\, v_\alpha(\varphi).
$$
\end{lemma}

\begin{proof} We have
\begin{align*}
&v_\alpha(Q_{K}(\varphi)) = \sup_{\substack  {{\omega_0 \in \Sigma} \\    \hat{v}_1 \neq \hat{v}_2}} \frac{\left | \int \varphi (\omega_1, \hat{A}(\omega_1,\omega_0)\hat{v}_1) -  \varphi(\omega_1,\hat{A}(\omega_1,\omega_0)\hat{v}_2)\;dK_{\omega_0}(\omega_1) \right|}{\delta(\hat{v}_1,\hat{v}_2)^\alpha}   \\
&\leq \sup_{\substack  {{\omega_0 \in \Sigma} \\    \hat{v}_1 \neq \hat{v}_2}} \frac{\int \left |  \varphi (\omega_1, \hat{A}(\omega_1,\omega_0)\hat{v}_1) -  \varphi(\omega_1,\hat{A}(\omega_1,\omega_0)\hat{v}_2) \right| d|K_{\omega_0}|(\omega_1)}{\delta(\hat{v}_1,\hat{v}_2)^\alpha}  \\
&\leq v_\alpha(\varphi) \sup_{\substack  {{\omega_0 \in \Sigma} \\    \hat{v}_1 \neq \hat{v}_2}} \int \frac{\delta(\hat{A}(\omega_1,\omega_0)\hat{v}_1, \hat{A}(\omega_1,\omega_0)\hat{v}_2)^\alpha}{\delta(\hat{v}_1,\hat{v}_2)^\alpha} \;d|K_{\omega_0}|(\omega_1) \\
&\leq v_\alpha(\varphi) k_\alpha(F_{A,K}).
\end{align*}
To conclude the lemma it is sufficient to notice that $Q^n_{K} = Q_{K^n}$. 
\end{proof}


We proceed to prove theorem \ref{analiticity TV Markov}. Let $K_0 \in \mathcal{K}_{\Prob}(\Sigma)$ be uniformly ergodic, assume that $L_1(K_0) > L_2(K_0)$ and that $K_0$ is is quasi-irreducible.

 We say that $K_0$ is quasi-irreducible if the associated Markov cocycle $F_A$ is quasi-irreducible. In other words, there is no proper invariant section $\mathcal{V} \colon \Sigma \to \Gr(\R^d)$ such that the top Lyapunov exponent restricted to it is not maximal (the Lyapunov exponent is defined almost everywhere with respect to the Markov measure $\Pp_{K_0, \mu_0}$ with initial distribution given by the unique stationary measure $\mu_0$ and transition kernel $K_0$).

Note that if $K_0$ is quasi-irreducible, then Kifer's non-random filtration is trivial (see \cite[Corollary 3.8]{CDKM1}). In particular, for all $v \in \Proj \backslash\{ 0 \}$ and for $\Pp_{K_0,\mu_0}$-a.e $\om$, 
$$
\lim_{n\to \infty} \frac{1}{n} \log \norm{A^n(\om)v} = L_1(A,K_0).
$$

Moreover, by \cite[Theorem 3.5]{CDKM1} we also have that 
\begin{equation} \label{unif markov}
\lim_{n\to \infty} \frac{1}{n} \int \log \norm{A^n(\om)v} \; d\Pp_{\om_0} =L_1(A,K_0),
\end{equation}
with uniform convergence in $(\om_0,\hat{v}) \in \Sigma \times \mathbb{S}^{d-1}$.

Furthermore, we also have the exponential contraction of $k_\alpha$ and $v_\alpha$ semi-norm for nearby kernels. 

\begin{proposition} \label{contraction of Valpha Markov}
Assume that $(A,K_0)$ is quasi-irreducible and that $L_1(A,K_0)>L_2(A, K_0)$. Then, there exists $0< \alpha \leq 1$, $\theta >1$, $C>0$ and a neighbourhood $V \subset \mathcal{K}(\Sigma)$ of $K_0$ with respect to the norm $\|.\|$, such that for every $n \in \N$ and for every $K \in V$,
\begin{equation} \label{eq contraction of Valpha Markov}
v_\alpha (Q_K^n (\varphi)) \le C \theta^{-n} v_{\alpha} (\varphi). 
\end{equation}
\end{proposition}

\begin{proof}
Since $K_0 \in \mathcal{K}_{\Prob}(\Sigma)$, it follows by \cite[Proposition 4.4]{CDKM1} that there exists $0< \alpha \leq 1$ and $n_0 \in \N$ such that $k_\alpha(F_{A^{n_0},K_0^{n_0}}) < \sigma <1$. Using the sub-multiplicative property of $k_\alpha$, we conclude that there exists $\theta= \sigma^{-\frac{1}{n_0}}>1$ and $C>0$ such that $k_\alpha(F_{A^n,K_0^n})< C \theta^{-n}$ for every $n \in \N$.

For a fixed $n$, the quantity  $\int \left({\frac{s_1( {A^n}) s_2( {A^n})}{\|{A^n} {v}\|^2}}\right)^\alpha \, d|\Pp_{\om_0}|$, which bounds  $k_\alpha (F_{A,K}^{n})$ from above, depends continuously on the kernel.  
Then, the inequality extends to a neighborhood of $K_0$, that is, there exists a neighbourhood $V \subset \mathcal{K}(\Sigma)$ of $K_0$ with respect to the norm $\|.\|$, such that $k_\alpha(F_{A^n,K^n}) \leq C \theta^{-n}$ for every $n \in \N$. Then we conclude the proof by applying lemma \ref{Q contracts Valpha Markov}.
\end{proof}







Similarly to the Bernoulli case, a consequence of proposition \ref{contraction of Valpha Markov}  is that, when $K \in \mathcal{K}_{\Prob} (\Sigma)$, there exists $\alpha \in (0,1]$, $\theta >1$ and $C<\infty$ such that for every $n \in \N$ and every $\varphi \in C^\alpha(\Proj)$,
\begin{equation} \label{eq. strong mixing Markov}
\norm{Q_{K}^n \varphi - \int \varphi \; d\eta_{K}}_\alpha \leq C\theta^{-n} \norm{\varphi}_\alpha,
\end{equation}
where $\eta_K \in \Prob(\Sigma \times \Proj)$ is the unique $\Pp_{K,\mu}$-stationary measure. The proof of this fact is the same as in the Bernoulli case. 

Consider the observable $\varphi \colon \Sigma \times \Proj \to \R$ given by
$$
\varphi(\om_0, \hat{v}) = \int_\Sigma \log \frac{\norm{A(\om_1, \om_0)v}}{\norm{v}} \; dK_{\om_0}(\om_1).
$$

When $K \in \mathcal{K}_{\Prob}(\Sigma)$, Furstenberg's formula says that,
$$
\int_{\Sigma \times \Proj} \varphi(\om_0, \hat{v}) \; d\eta_{K}(\om_0, \hat{v}) = L_1(A,K).
$$ 

In this case, for a fixed $\hat{v} \in \Proj$ the iterates $Q_K^n \varphi (\hat{v})$ converge to the top Lyapunov exponent $L_1(K)$.

Denote by $\mathcal{K}_0(\Sigma)$ the set of (continuous) complex valued  Markov Kernels such that for every $\om \in \Sigma$, $K_\om$ is a complex measure satisfying $K_\om(\Sigma)=0$. Note that $\mathcal{K}_0(\Sigma)$ is a closed subspace of $\mathcal{K}(\Sigma)$, therefore it is also a Banach space.

Denote by $\mathcal{K}_1 (\Sigma)$ the set of complex Markov kernels such that for every $\om \in \Sigma$, $K_\om(\Sigma)=1$. Then, $\mathcal{K}_1$ is an affine Banach space and we can endow it with an analytic structure. 

\begin{proposition} \label{proposition uniform convergence Markov}
For every $K \in \mathcal{K}_1(\Sigma)$ sufficiently close to $K_0$, $\hat{v_1} \neq \hat{v_2} \in \Proj$ and $\varphi \in C^\alpha(\Sigma \times \Proj)$, 
$$
\left |  Q^{n+1}_{K}\varphi(\hat{v}_1) - Q^n_{K}\varphi(\hat{v}_2)   \right| \leq C v_\alpha (\varphi) \theta^{-n} \norm{K}.
$$
\end{proposition}

\begin{proof}
First note that by lemma \ref{Q contracts Valpha Markov}, for every $K \in \mathcal{K}(\Sigma)$ and every $\varphi \in C^\alpha(\Sigma \times \Proj)$, we have $v_\alpha(Q^n_{K}\varphi) \leq v_\alpha(\varphi)k_\alpha(F_{A^n,K^{n}})$.

Therefore, it also holds that for every $\hat{v}_1 \neq \hat{v}_2 \in \Proj$,
$$
\left |  Q^n_{K}\varphi(\hat{v}_1) - Q^n_{K}\varphi(\hat{v}_2)   \right| \leq v_\alpha(\varphi)k_\alpha(F_{A^n,K^{n}}).
$$

Then, by proposition \ref{contraction of Valpha Markov}, we conclude that for every $\hat{v_1} \neq \hat{v_2} \in \Proj$, $\varphi \in C^\alpha(\Sigma \times \Proj)$ and for $K \in \mathcal{K}(\Sigma)$ sufficiently close to $K_0$ 
\begin{equation} \label{EqContracao1}
\left |  Q^n_{K}\varphi(\hat{v}_1) - Q^n_{K}\varphi(\hat{v}_2)   \right| \leq C v_\alpha (\varphi) \theta^{-n}.
\end{equation}

Note that for every $\hat{v}_1 \neq \hat{v}_2 \in \Proj$,
$$
\left| Q^{n+1}_{K} \varphi (\hat{v}_1) - Q^{n}_{K} \varphi (\hat{v}_2)    \right| =  \left |  \int_{\Sigma} Q^n_{K}\varphi(\hat{A}(\om_1, \om_0)\hat{v}_1) \; d K_{\om_0} (\om_1) -  Q^{n}_{K} \varphi (\hat{v}_2)  \right |.
$$

Moreover, for every $K \in \mathcal{K}_1(\Sigma)$ sufficiently close to $K_0$ and $\hat{v_1} \neq \hat{v_2} \in \Proj$, it holds that 
\begin{align*}
& \left |  \int_{\Sigma} Q^n_{K}\varphi(\hat{A}(\om_1,\om_0)\hat{v}_1)  \; dK_{\om_0}(\om_1) -  Q^{n}_{K} \varphi (\hat{v}_2)  \right | = \\
= & \left |  \int_{\Sigma} Q^n_{K}\varphi(\hat{A}(\om_1, \om_0)\hat{v}_1)  -  Q^{n}_{K} \varphi (\hat{v}_2)  \; dK_{\om_0}(\om_1)  \right | \leq  \\
&\leq \int_{\Sigma} \left | Q^n_{K}\varphi(\hat{A}(\om_1, \om_0)\hat{v}_1)  -  Q^{n}_{K} \varphi (\hat{v}_2) \right |  \; d |K_{\om_0}|(\om_1) \leq \\
&\leq C v_\alpha (\varphi) \theta^{-n} \norm{K}.
\end{align*}
\end{proof}

Therefore, the limit function $K \mapsto Q_K^\infty \varphi(\hat{v})$ exists and is continuous, because it is a uniform limit of continuous functions. 

Similarly to the Bernoulli case, we note that proposition \ref{contraction of Valpha Markov} could have been done directly in $\mathcal{K}_1(\Sigma)$. Thus, we consider the neighbourhood $V \subset \mathcal{K}_1(\Sigma)$ from that proposition for what follows.

Consider kernels of the form $K + zL$, where $K \in V \subset \mathcal{K}_1(\Sigma)$  and $L \in \mathcal{K}_0(\Sigma)$. For a fixed $\hat{v}$, the iterates $Q^n_{K+ zL} \varphi(\hat{v})$ form a polynomial in $z$ with degree less or equal to $n$.
Thus, for every $z \in V(K, L)=\{z \in \C : K+zL \in V\}$, the map $z \mapsto Q^\infty_{K+zL} \varphi(\hat{v}) $ is a uniform limit of holomorphic functions, hence it is holomorphic. 

Thus, we conclude that the Lyapunov exponent is G-holomorphic in a neighbourhood $V \subset \mathcal{K}_1(\Sigma)$ of $K_0$. Together with continuity, this implies that it is holomorphic. 

Now we proceed to prove item $2$ of theorem \ref{analiticity TV Markov}.
For any fixed kernel $K$, we may define the set of continuous kernels whose supports are contained in the support of $K$ as follows
$$
\mathcal{S}(K) = \{ L \in \mathcal{K}(\Sigma) \colon \supp ({L}) \subseteq \supp ({K}) \}.
$$

Note that, endowed with the norm defined above, $\mathcal{S}(K)$ is a Banach space. Similarly to what we did before, we consider the sets
$$
\mathcal{S}_0(K) = \{ L \in \mathcal{K}(\Sigma) \colon \supp ({L}) \subseteq \supp ({K}) \text{ and } L_\om(\Sigma)=0 \; \forall \om \}
$$
and
$$
\mathcal{S}_1(K) = \{ L \in \mathcal{K}(\Sigma) \colon \supp ({L}) \subseteq \supp ({K}) \text{ and } L_\om(\Sigma)=1 \; \forall \om \}.
$$

Note that $\mathcal{S}_0(K)$ is also a Banach space and $\mathcal{S}_1(K)$ is a translation of a Banach space, therefore we may endow it with an analytic structure. Moreover, the proof of item $1$ also holds similarly when we restrict the ambient domain to be $\mathcal{S}_1(K_0)$ instead of $\mathcal{K}_1(\Sigma)$. 

Now, we drop the assumption of irreducibility of ${K}_0$. Let $\mathcal{V} \colon \Sigma \to \Gr(\R^d)$ be a non trivial invariant section. 
Then the corresponding vector sub-bundle
$$
\mathbb{V} = \{ (\om, v)\colon \om \in \Sigma^\N , v \in \mathcal{V}(\om_0) \}
$$
is $F_A$ invariant, hence we can restrict the cocycle to the sub-bundle $\mathbb{V}$. Therefore we can consider the induced cocycle $F_{\mathcal{V}} \colon \mathbb{V} \to \mathbb{V}$. Similarly, we can consider the quotient vector bundle
$$
\Sigma^\N \times \R^d / \mathbb{V} \colon = \bigcup_{\om \in \Sigma^\N} \{\om\} \times \R^d / \mathcal{V}(\om_0).
$$

Also, there is the quotient cocycle $F_{\R^d/\mathcal{V}}$ on the vector bundle $\Sigma^\N \times \R^d / \mathbb{V}$. We denote by $L_1(K_{0}, \mathcal{V})$ and $L_1(K_{0}, \R^d / \mathcal{V})$ their top Lyapunov exponents, respectively. Note that 
\begin{equation}\label{eq Lyapdecomp}
L_1(K_0) = \max \{ L_1({K_{0}, \mathcal{V}}), L_1({K_{0}, \R^d / \mathcal{V}} ) \}.
\end{equation} 

Without loss of generality we may assume that $L_1(K_0) = L_1(K_{0}, \mathcal{V})$. The other case is similar. Note that for every $K$ in a neighbourhood $V \subset \mathcal{S}_1(K_0)$ of $K_0$, it holds that $\supp K \subseteq \supp K_0$, which implies that $\mathcal{V}$ is also invariant for $K$.

Remember that in the proof of item $2$ of Theorem \ref{analiticity TV Bernoulli} we used corollary B of \cite{Pe91}, which follows from \cite[Lemma 3.6]{FK83} and \cite[Theorem B]{FK83} to conclude that, in fact, if we assume that $L_1(\mu_0) = L_1(\mu_W)$, then $L_1(\mu) = L_1(\mu_W)$ for every $\mu$ sufficiently close to $\mu_0$. Equation \ref{eq Lyapdecomp} is the Markovian analogue of lemma 3.6 of \cite{FK83}. For a proof of it, see \cite[Proposition 3.5]{CDKM1}. A Markovian analogue of \cite[Theorem B]{FK83} follows the same proof from the original paper, using the Markovian version of Furstenberg's Formula, available at \cite[Theorem 3.3]{CDKM1}. Hence, 
if $K_n \to K_0$ and $\supp (K_n)_\om \subseteq \supp (K_0)_\om$ for every $n$ and $\om$, then $L_1(K_n) \to L_1(K_0)$.

Therefore, we are able to use the continuity of the Lyapunov exponents to each of the induced cocycles and conclude that 
for every $K$ sufficiently close to $K_0$, it holds that  $L_1(K) = L_1(K,{\mathcal{V}})$. 
If $K_{0}$ is irreducible for $F_{\mathcal{V}}$, then the map $K \mapsto L_1(K, {\mathcal{V}})$ is holomorphic. Since $L_1(K) = L_1(K, {\mathcal{V}})$ in a neighbourhood of $K_0$, we conclude that $K \mapsto L_1(K)$ is also holomorphic in a neighbourhood of $K_0$.

If $ K_{0}$ is not irreducible for $F_{\mathcal{V}}$, there exists another non trivial invariant section $\mathcal{V}' \subset \mathcal{V}$. Hence consider the induced cocycles $F_{\mathcal{V}'}$ and $F_{\mathcal{V} / \mathcal{V}'}$. Then we do the same procedure. Since the invariant sections are of decreasing dimension, this process must stop after a finite number of steps. Therefore, we conclude the proof of theorem \ref{analiticity TV Markov}.

\subsection*{Acknowledgments} 
We are most grateful to A. Cai, S. Klein and M. Viana for the discussions and numerous suggestions on the text. A.A. was supported by a CAPES doctoral fellowship. M.D. was supported by a CNPq doctoral fellowship and A.M. was partially supported by Instituto Serrapilheira, grant ``Jangada Dinâmica: Impulsionando Sistemas Dinâmicos na Região Nordeste'' and CNPq-Brazil-Bolsa de Pós-doutorado Júnior No. 152968/2024-5.

This study was financed in part by the Coordenação de Aperfeiçoa-
mento de Pessoal de Nível Superior – Brasil (CAPES) – Finance Code 001.
\bigskip

\bibliographystyle{amsplain}
\bibliography{references}

\end{document}